\documentclass[11pt]{elsarticle}
\usepackage{amssymb}
\usepackage{xcolor}
\usepackage{mathrsfs}
\usepackage{amscd}
\usepackage{appendix}
\usepackage{geometry}
\usepackage{setspace}
\usepackage{amsfonts}
\usepackage{amsmath}
\usepackage{lineno}
\usepackage{hyperref}
\usepackage{tikz}

\usetikzlibrary{matrix,arrows}
\newtheorem{definition}{Definition}[section]
\newtheorem{theorem}[definition]{Theorem} 
\newtheorem{lemma}[definition]{Lemma}     
\newtheorem{coro}[definition]{Corollary}

\newtheorem{prop}[definition]{Proposition}
\newtheorem{remark}[definition]{Remark}
\newtheorem{Problem}[definition]{Problem}

\numberwithin{equation}{section}
\def\Proof{\\\noindent \it Proof.\ \ \rm}
\def\qedbox{$\rlap{$\sqcap$}\sqcup$}
\biboptions{numbers,sort&compress} 
\def\Proof{
\noindent \it Proof.\ \ \rm}
\def\qedbox{$\rlap{$\sqcap$}\sqcup$}

\newcommand{\Z}{{\mathbb Z}}
\newcommand{\R}{{\mathbb R}}

\newcommand{\T}{{\mathbb T}}

\newcommand{\N}{{\mathbb N}}

\newcommand{\Leb}{{\mathrm{Leb}}}

\journal{arXiv}

\begin{document}

\begin{frontmatter}

\title{Anderson Localization for Schr\"{o}dinger Operators with\\ Monotone Potentials Generated by the Doubling Map}

\author[]{Yuanyuan Peng}
\ead{lunarpeng@foxmail.com }

\author[]{Chao Wang}
\ead{cwangmath@sohu.com }


\author[]{Daxiong Piao\corref{cor1}}
\ead{dxpiao@ouc.edu.cn}
\cortext[cor1]{Corresponding author}

\address{School of Mathematical Sciences,  Ocean University of China,
Qingdao 266100, P.R.China}

\begin{abstract}

In this paper, we consider the Schr\"{o}dinger  operators  on $ \ell^{2}(\N) $, defined for all $ x\in\mathbb{T}  $ by
\begin{equation}
(H(x)u)_n = u_{n+1} + u_{n-1}  + \lambda f(2^{n} x) u_n, \quad \text{for } n \geq 0,\notag
\end{equation}
with the Dirichlet boundary condition $ u_{-1}=0 $. Building on Zhang's recent breakthrough work [Comm.Math.Phys.405:231(2024)] that resolved Damanik's open problem [Proc.Sympos. Pure Math.76,Amer.Math.Soc.(2007)]  on the uniform positivity of the Lyapunov exponent,  for the potential $ f \in C^{1}(0,1)$  with $ \|f\|_{C^{1}(0,1)} < C $ and $ \inf_{x \in (0,1)} |f^{\prime}(x)| > c>0 $, we  obtain the large deviation estimate and prove that for a.e. $ x \in \mathbb{T} $ and sufficiently large $ \lambda > \lambda_{0} $, the operators $ H(x) $ display Anderson localization. Furthermore, if the potentials  also have zero mean, our analysis reveals that the doubling map models can exhibit localization behavior for both  small and large coupling constants $ \lambda $.

\end{abstract}

\end{frontmatter}

\tableofcontents
\section{Introduction}

The spectral analysis of Schr\"{o}dinger operators with deterministic potentials generated by hyperbolic dynamical systems represents a fundamental topic at the intersection of dynamical systems and mathematical physics. The chaotic nature of such systems endows the potentials with pseudo-random characteristics, leading to expectations of positive Lyapunov exponents and Anderson localization—a phenomenon characterized by pure point spectrum and exponentially decaying eigenfunctions, which corresponds to suppressed quantum transport in physical systems.

The doubling map $T(x) = 2x \mod 1$ on $\mathbb{T} = \mathbb{R}/\mathbb{Z}$, with its Bernoulli structure and strong mixing properties, serves as a paradigmatic model for studying this interplay between deterministic dynamics and localization. Fundamental work by Damanik and Killip \cite{D05} established that for any bounded, measurable, non-constant sampling function $f$, the Lyapunov exponent $L(E)$ is positive for almost every energy $E$, confirming the random-like character of this model. This prompted Damanik \cite{D07} to pose a deeper quantitative question:
\begin{Problem}
Find a class of functions $f\in L^\infty(\mathbb{T})$ such that for large coupling constants $\lambda > 0$, the Lyapunov exponent of the doubling map model satisfies $\inf_E L(E) > c \log \lambda$.
\end{Problem}

A breakthrough was recently achieved by Zhang \cite{Z24}, who answered this question affirmatively for monotone $C^1$ potentials with non-vanishing derivatives. His work, building on the framework of Young's hyperbolic cocycles \cite{Y} and the polar coordinate representation of Schr\"{o}dinger cocycles \cite{WZ,Z12}, effectively quantifies the competition between the hyperbolicity of the base dynamics and the cocycle. However, Zhang's class of potentials introduces a significant novelty: they may exhibit jump discontinuities at the periodic endpoints, requiring the completion $f(0) = \lim_{x\to 0^+} f(x)$, which represents a departure from the globally $C^1$ potentials typically considered in earlier work.

While the uniform positivity of the Lyapunov exponent established by Zhang is crucial, it alone cannot guarantee Anderson localization due to the potential occurrence of ``double resonances" that may prevent exponential decay of eigenfunctions and lead to singular continuous spectrum. The pioneering work of Bourgain and Schlag \cite{BoS} provides a comprehensive framework for establishing localization through three key components: large deviation estimates, exponential bounds on Green's functions, and double resonance analysis. However, their approach is fundamentally perturbative, relying on small coupling constants $\lambda$ and globally $C^1$ potentials.

\textbf{Our Contributions.} This paper bridges the gap between Zhang's Lyapunov exponent bounds and complete Anderson localization proof for large coupling constants. We extend and adapt the Bourgain-Schlag framework to accommodate both large $\lambda$ regimes and Zhang's class of potentially discontinuous monotone potentials. Our main innovations include:

\begin{enumerate}[{(1)}]
\item \textbf{Handling Non-Smooth Potentials:} We develop techniques to manage potentials that may lack global $C^1$ smoothness, particularly those with endpoint discontinuities. By leveraging monotonicity and ergodic properties of the doubling map, we overcome challenges not addressed in previous frameworks.

\item \textbf{Non-Perturbative Large Deviation Estimates:} Under Zhang's monotonicity assumption and relaxed regularity conditions, we establish large deviation estimates in the large-$\lambda$ regime. Our approach, based on strong mixing properties and refined analysis of angle evolution, demonstrates robustness against potential singularities.

\item \textbf{Parameter-Sensitive Quantitative Control:} We carefully trace $\lambda$-dependencies throughout the estimation process, showing that key quantities remain controllable for large $\lambda$ when the Lyapunov exponent is sufficiently large.

\item \textbf{Dual-Parameter Scaling Strategy:} We introduce new scale selection methods to maintain near-independence properties when the Lyapunov exponent scales as $O(\log \lambda)$.

\end{enumerate}

Our work demonstrates that the Bourgain-Schlag framework, when combined with Zhang's Lyapunov exponent estimates and enhanced to handle singular potentials, can establish localization in previously inaccessible regimes. This significantly expands our understanding of how deterministic disorder leads to localization in quantum systems.

\emph{Organization of the paper}. Section 2 introduces the setting and states the main results. Section 3 presents the analysis of the Schr\"{o}dinger cocycle in polar coordinates. The key technical estimates are addressed in the subsequent sections:   the large deviation estimate is proven in Section 4, the H\"{o}lder continuity of the Lyapunov exponent is established in Section 5,  Section 6 establishes estimates for the truncated part of the Green's function, and the measure of the double resonances set is derived in Section 7. Finally, Section 8 completes the proof of Anderson localization  by eliminating the double resonance set.

\section{Setting and Main Results}

We begin by introducing the setting and  main results. Detailed contents can be found in  \cite{D05,D23,D,Z24}.

Consider the discrete Schr\"{o}dinger operator $H(x): \ell^2(\N) \to \ell^2(\N)$, defined for all $ x\in\mathbb{T}  $ by
\begin{equation}\label{2.1}
(H(x) u)_n= u_{n+1} + u_{n-1} + \lambda f(T^n x) u_n, \quad  \text{for}\,\, n\geq 0,
\end{equation}
with Dirichlet boundary condition $u_{-1}= 0$. Here, $\lambda \in \mathbb{R}$ denotes the coupling constant, and the potential  is generated by the doubling map $T: \mathbb{T} \to \mathbb{T}$, $x \mapsto 2x \mod 1$ with $f: \mathbb{T} \to \mathbb{R}$ assumed to be bounded, measurable and non-constant.  This defines an  ergodic family of operators $\{H(x)\}_{x \in \mathbb{T}}$, see \cite[Chapter 3]{D}.

For any energy $E\in\R$, the transfer matrix $A(x,E): \mathbb{T} \to \mathrm{SL}(2,\mathbb{R})$ is defined by
\begin{equation}\label{eq:trans-matrix}
A(x,E) = \begin{pmatrix}
E - \lambda f(x) & -1 \\
1 & 0
\end{pmatrix},
\end{equation}
and the Schr\"{o}dinger cocycle $(T, A(x,E)): \mathbb{T} \times \mathbb{R}^2 \to \mathbb{T} \times \mathbb{R}^2$ is given by
$(x, v) \mapsto (Tx, A(x,E)v)$.
We define the  cocycle  $(T, A(x,E))^n = (T^n, A_n(x,E))$ for $n\geq1$, where
\begin{equation}\label{a}
A_n(x,E) =
 A(T^{n}x,E)A(T^{n-1}x,E) \cdots A(Tx,E) .
\end{equation}

The eigenvalue equation $H(x)u = Eu$ is equivalent to
\begin{equation}
u_{n+1} + u_{n-1} + \lambda f(T^nx)u_n = Eu_n, \quad n \in \N.
\end{equation}
For any initial vector $\boldsymbol{u}_0=(u_0, u_1)^T \in \mathbb{R}^2$, the solution satisfies
\begin{equation}
\begin{pmatrix} u_{n+1} \\ u_n \end{pmatrix}= A_n(x,E)\begin{pmatrix} u_1 \\ u_0 \end{pmatrix}, \quad n \in \N.
\end{equation}
Thus, the transfer matrices $A_n(x,E)$ govern the evolution of solutions.

The Lyapunov exponent for the ergodic system is  given by
\begin{equation}\label{2.6}
L(E) = \lim_{n \to \infty} \frac{1}{n} \int_{\mathbb{T}} \log \|A_n(x, E)\|  dx \geq 0,
\end{equation}
where $\| \cdot \|$ denotes the operator norm. By Kingman's Subadditive Ergodic Theorem, it follows that
\begin{equation}
L(E) = \lim_{n \to \infty} \frac{1}{n}  \log \|A_n(x, E)\|
\end{equation}
for a.e. $x\in\T$.

Throughout this paper,  let $C$ and $c$  denote universal positive constants, with $C$ representing larger constants and  $c$ smaller ones.

Furthermore, we  introduce the notation: For any $a, b > 0$, $a \sim b$ means that $c a \leqslant b\leqslant C a $.

First, we introduce the results on generalized eigenvalues and generalized eigenfunctions, which are essential for our analysis of  Anderson localization.

\begin{lemma}\cite[Theorem 2.4.4]{D05}\label{th0}
A nontrivial sequence $u(x) = \{u_n(x)\}_{n\in\N}$ is a generalized eigenfunction of $H(x)$  if $u(x)$ solves the eigenvalue equation $H(x) u(x) = E u(x)$ for some $E$ and satisfies
\begin{equation} \label{eq:2.4.1}
|u_n(x)| \leq C(1 + |n|)^{\delta}
\end{equation}
for suitable finite constants $C, \delta > 0$, and every $n \in \N$.  Let $\mathcal{E}_{\delta} = \mathcal{E}_{\delta}(H)$ denote the set of generalized eigenvalues, which are the energies  $E$  satisfying  \eqref{eq:2.4.1}. Let $\mathcal{E} =\mathcal{E}(H) = \bigcup_{\delta > 0} \mathcal{E}_{\delta}$, then the spectrum satisfies $\sigma(H(x))=\bar{\mathcal{E}}$.
\end{lemma}

 According to Shnol's work \cite{Sh} and Lemma~\ref{th0},  we have following result for $H (x)$.

 \begin{lemma}\label{r}
 Suppose that for  any generalized eigenvalue $E$ of $H (x)$, the associated generalized eigenfunction decays exponentially:
 \begin{equation}
\left|u_{n}\right|  \leq e^{-cn}, \quad n \in \mathbb{\N}.
\end{equation}
 Then, $E$ is an eigenvalue of $H(x)$ and $u_n(x$) is the corresponding  eigenfunction. Consequently,  the operator $H_\lambda(x)$ has pure point spectrum  and  all eigenfunctions decay exponentially at infinity, i.e. the operator $H(x)$ exhibits Anderson localization.
\end{lemma}

Now we can state our main results as follows:

\begin{theorem}[H\"{o}lder continuity of $L(E)$]{\label{th01}}
For sufficiently large $\lambda$,  and for energies $E_{1},E_{2}\in\sigma(H(x))$, there exist suitable positive constants $C_\lambda$ and $c$, such that

\begin{equation}
|L(E_{1})-L(E_{2})|<C_\lambda|E_{1}-E_{2}|^{\frac{c}{{(\log\lambda)^3}} }.
\end{equation}
\end{theorem}

\begin{theorem}[Localization for large coupling] \label{th1}
Consider a monotonic function $f: \T\to \mathbb{R}$ which is $C^1$ on $(0,1)$. Assume that:

\begin{itemize}
    \item On the torus $\mathbb{T}$, $x = 0$ is a jump discontinuity point of $f$.
    \item $f$ is right-continuous at $0$, that is $f(0)=\lim_{x\to 0^+}f(x)$.
    \item The right derivative $f_{+}^{\prime}(0)$ exists and is finite.
    \item The derivative satisfies $\inf_{x\in(0,1)}|f^{\prime}(x)|\geq c>0$.
    \item The $C^1$ norm on $(0,1)$, $\|f\|_{C^{1}(0,1)}<C$.
\end{itemize}
 Here, $c$ and $C$ are positive constants depending on $f$.
Then there exists a constant $\lambda_0=\lambda_0(f)>0$ such that for all $\lambda>\lambda_0$ and a.e. $x\in\T$, the  operator \eqref{2.1}  exhibits
 Anderson localization.
\end{theorem}

Without loss of generality, we assume the function $f$ satisfies $f(0)=0$, $\lim_{x\rightarrow 1{-}}f(x)=1$, right derivative $f'_+(0)=1$ and $f'(x)\in(c,1]
$ for  $x\in(0,1)$. For estimates that hold uniformly for sufficiently large $\lambda$ or for all ${E}\in\sigma(H(x))\subseteq  [-2\lambda,2\lambda]$, we will  leave the dependence on $\lambda$ or  $E$ implicit.

\section{ The Schr\"{o}dinger Cocycle in  Polar Coordinates}
We define a function
\begin{equation}\label{b01}
g(x)\doteq({E}/{\lambda}-f(x))^2+1,
\end{equation}
where ${E}\in\sigma(H(x))\subseteq  [-2\lambda,2\lambda]$.
 Subsequently, we  introduce a new transfer matrix  $B(x)=B(x,E)$ and the corresponding cocycle $(T, B(x )): (x, v) \mapsto (Tx, B(x )v)$ on $\mathbb{T} \times \mathbb{R}^2$.
The cocycle is defined as $(T, B(x ))^n = (T^nx, B_n(x ))$, where
\begin{equation}\label{b}
B_n(x ) =
B(T^{n-1}x )B(T^{n-2}x ) \cdots B(x ).
\end{equation}

\begin{lemma}\label{le1}
There exists $\lambda_0=\lambda_0(f)>0$ such that for all $\lambda>\lambda_0$, $x\in\T$  and $E\in [-2\lambda,2\lambda]$,  the transfer matrix $B(T^nx )$ $(n\geq0)$ admits
\begin{align}
B(T^nx )&=\Lambda(T^{n+1}x )\cdot R_{{\theta}(T^nx )}\notag\\
&=
\left(
\begin{matrix}
\lambda\sqrt{g(T^{n+1}x )} & 0\\
0 & (\lambda\sqrt{g(T^{n+1}x )})^{-1}
\end{matrix}
\right)\cdot
\left(
\begin{matrix}
\frac{{E}/{\lambda}-f(T^nx)}{\sqrt{g(T^{n}x )}} & \frac{-1}{\sqrt{g(T^{n}x )}}\\
\frac{1}{\sqrt{g(T^{n}x )}} & \frac{{E}/{\lambda}-f(T^nx)}{\sqrt{g(T^{n}x )}}
\end{matrix}
\right),\label{b02}
\end{align}
where
\begin{equation}\label{b03}
\cot{\theta}(T^nx )={\frac{{E}/{\lambda}-f(T^nx)}{\sqrt{g(T^{n}x )}}}\Big/  {\frac{1}{\sqrt{g(T^{n}x )}}}={E}/{\lambda}-f(T^nx).
\end{equation}
 Moreover,  the Lyapunov exponent $L_B(E)$ of the cocycle $(T,B(x))$ satisfies
\begin{equation}\label{lbl}
L_B(E)= L(E).
\end{equation}
\end{lemma}

\Proof
Following the approach in \cite[Appendix A.1]{WZ} and \cite[Appendix A]{Z24}, we convert the original Schr\"{o}dinger cocycle into its equivalent  polar coordinate representation. Consider the transfer matrix $A(x,E)$ defined in \eqref{eq:trans-matrix}. For large  $\lambda >0 $, we apply a diagonal similarity transformation with
\begin{equation}
Q = \begin{pmatrix} \sqrt{\lambda}^{-1} & 0 \\ 0 & \sqrt{\lambda} \end{pmatrix}
\end{equation}
to obtain the rescaled matrix:
\begin{equation}
\tilde{A}(x,E) = Q A(x,E) Q^{-1} = \begin{pmatrix} \lambda [\frac{E}{\lambda} - f(x)] & -1/\lambda \\ \lambda & 0 \end{pmatrix}.
\end{equation}
 This rescaling preserves the Lyapunov exponent since $Q$ is diagonal with constant determinant.

For each $x \in \mathbb{T} \setminus \{0\}$ (where $f$ is $C^1$), we perform the polar decomposition of $\tilde{A}(x,E)$. As shown in \cite[Appendix A.1]{WZ} and \cite[Appendix A]{Z24}, there exist orthogonal matrices $U_1(x), U_2(x) \in \mathrm{SO}(2,\mathbb{R})$ and a diagonal matrix
\begin{equation}
\Lambda(x) = \begin{pmatrix}  \lambda\sqrt{g(x )} & 0 \\ 0 & (\lambda\sqrt{g(x )})^{-1} \end{pmatrix}
\end{equation}
such that:
\begin{equation}
\tilde{A}(x,E) = U_1(x) U_2(x) \Lambda(x) U_2^T(x).
\end{equation}
Define $U(x) = U_1(x) U_2(x)$. Then we have
\begin{equation}
U^{-1}(Tx) \tilde{A}(Tx,E) U(x) = \Lambda(Tx) \cdot R_{\theta(x)},
\end{equation}
where $R_{\theta(x)}=U^T_2(Tx) (U_1U_2)(x)$ is a rotation matrix and expressed as
\begin{align}
     R_{\theta(x)}=\left(\begin{matrix}
c(x,t,\lambda,T) & -\sqrt{1-c^2(x,t,\lambda,T)}\\
\sqrt{1-c^2(x,t,\lambda,T)} & c(x,t,\lambda,T)\end{matrix}\right),
\end{align}
where $t=\frac{E}{\lambda}\in[-2,2]$. For sufficiently large $\lambda$, we replace $c(x,t,\lambda,T)$ by $c(x,t,\infty ,T)=\frac{t-f(x)}{\sqrt{(t-f(x))^2+1}}$ ($T$ denotes the doubling map) since the validity of this replacement  is ensured by the $C^1$-closeness to the limiting case $\lambda\rightarrow \infty $, as detailed in \cite[Appendix A.1]{WZ} or \cite[Appendix A]{Z24}.
Hence, the explicit form of the right-hand side is the polar coordinate representation.

By defining
\begin{equation}
B(x)=U^{-1}(Tx)\tilde{A}(Tx,E)U(x),
\end{equation}
 we obtain
\begin{equation}
B(x) = \Lambda(Tx)\cdot R_{\theta(x)},
\end{equation}
which is exactly the form given in \eqref{b02}.

The conjugacy relation implies that for any $n \geq 1$:
\begin{equation}
B_n(x) = B(T^{n-1}x) B(T^{n-2}x)  \cdots B(x) =  U^{-1}(T^n x) QA_n(x)Q^{-1} U(x) ,
\end{equation}
where $A_n(x) = A(T^{n}x) A(T^{n-1}x)  \cdots A(Tx)$ is the original cocycle. Since $U(x)$ is orthogonal and $Q$ is constant-diagonal, we have the norm inequalities:
\begin{align}
\|B_n(x)\| \leq  \|U^{-1}(T^n x)\|\cdot \|Q\| \cdot \|A_n(x)\|  \cdot\|Q^{-1}\| \cdot \|U(x)\| \le\lambda \|A_n(x)\|, \label{eq:ab1}\\
\|A_n(x)\|  \leq \|Q^{-1}\|\cdot \|U(T^n x) \|\cdot \|B_n(x)\|\cdot \|U^{-1}(x)\|\cdot \| Q\| \le  \lambda  \|B_n(x)\|,\label{eq:ab2}
\end{align}
where we used $\|Q\| = \|Q^{-1}\| =\sqrt{\lambda}$. These imply cocycle $B_n(x)$ is the equivalent form of  original cocycle $A_n(x)$:
\begin{equation}\label{abn}
\frac{1}{\lambda} \|A_n(x)\| \leq \|B_n(x)\| \leq \lambda \|A_n(x)\|.
\end{equation}
Taking logarithms, averaging over $n$, and integrating over $x \in \mathbb{T}$, we obtain:
\begin{equation}
L_n^A(E) - \frac{\log \lambda}{n} \leq L_n^B(E) \leq L_n^A(E) + \frac{\log \lambda}{n},
\end{equation}
where $L_n^A(E) = \frac{1}{n} \int_{\mathbb{T}} \log \|A_n(x)\| \, dx$ and similarly for $L_n^B(E)$. Taking the limit $n \to \infty$, the terms $\frac{\log \lambda}{n} \to 0$, yielding $L_B(E) = L(E)$. This completes the proof of \eqref{lbl}. \hfill \qedbox

\medskip

Based on the formalism of Zhang \cite{Z24}, we  analyze the long-term behavior of the cocycle $(T,B(x ))$  to derive the properties of   $\boldsymbol{v}_n(x) = B_n(x)\boldsymbol{v}_1$, where the initial  vector  $\boldsymbol{v}_1= (\cos\beta , \sin \beta )^T $ is an arbitrary  unit vector.

After $n\geq1$ full iterations of cocycle in polar coordinates,   the total rotation angle  is denoted by $\phi_n(x)$. Omitting the scaling component at the $n$th step, the corresponding cumulative angle is denoted by $\theta_{n-1}(x)$.  The recurrence relations satisfied by the angles are as follows:
\begin{prop}\label{prop1}
For sufficiently large $\lambda$ and all  $E\in [-2\lambda,2\lambda]$, the rotation angle $\phi_n(x)$ $(n\geq1)$ and the auxiliary angle $\theta_n(x)$ $(n\geq0)$ satisfy
\begin{align}
&\theta_{n-1}(x)=\phi_{n-1}(x)+\theta_0(T^{n-1}x),\label{b07}
\\&\cot\phi_n(x)=\lambda^2g(T^{n}x)\cot\theta_{n-1}(x),\label{b06}
\end{align}
where $ \theta_{0}(x)={\theta}(x)+\beta$ and $ \theta_{n-1}(x),\phi_n(x):\T\rightarrow\R$.

\end{prop}

\Proof
Based on Lemma \ref{le1},  we apply $B(Tx)$  to the initial vector $\boldsymbol{v}_1= (\cos\beta , \sin \beta )^T $:
\begin{align}
B(x)\cdot\boldsymbol{v}_1&=\left(
\begin{matrix}
\lambda\sqrt{g(Tx)}\cos({\theta}(x)+\beta)\\
(\lambda\sqrt{g(Tx)})^{-1}\sin({\theta}(x)+\beta) \end{matrix}
\right). \label{prop4}
\end{align}
Hence, the  rotation angle $\phi_1(x)$  is
\begin{equation}
\cot\phi_1(x)=\lambda^2g(Tx)\cot({\theta}(x)+\beta).
\end{equation}
Let $ \theta_{0}(x)={\theta}(x)+\beta$. After  $n$ iterations,  we obtain \eqref{b07} and \eqref{b06} for all $n\geq1$.
\hfill \qedbox
\medskip

\begin{remark}\label{re1}
From the transfer matrix  $B(T^nx)$,  starting from  the corresponding solution vector with the  angle $\theta_{n-1}(x)$, we scale the  vector by multiplying its $x$-coordinate by $(\lambda\sqrt{g(T^{n}x )})$	and its $y$-coordinate by  $(\lambda\sqrt{g(T^{n}x )})^{-1}$. This transformation yields the rotation angle $ \phi_n(x)$.
Hence  the resulting vector does not cross the $x$-axis and $|\phi_n(x)-\theta_{n-1}(x)|\leq \pi$. Consequently, the rotation angle for the rotation angle $\phi_n(x)$ $(n\geq1)$ can be expressed as
\begin{equation}\label{prop2}
\phi_n(x)=\phi_n(x) =\begin{cases}
\mathrm{arccot}(\lambda^2g(T^{n}x)\cot\theta_{n-1}(x))+\tilde{n}\pi, & \theta_{n-1}(x) \notin   \pi\Z,\\
\tilde{n}\pi, & \theta_{n-1}(x) \in  \pi\Z,\end{cases}
\end{equation}
where $\tilde{n}=\lfloor\frac{\theta_{n-1}(x)}{\pi}\rfloor\in\Z$
accounts for the correct branch of the arccotangent function.
\end{remark}

We extend the range of values for ${E}/{\lambda}$ (i.e., $t$ in \cite{Z24}) while preserving the boundedness of the function $g$. Zhang's conclusions remain valid under this condition, and the proof follows the same method, which is omitted here for brevity. The Corollary 1 of \cite{Z24} is provided as follows.

\begin{lemma}\label{Z2}  For sufficiently large $\lambda$ and all  $E\in [-2\lambda,2\lambda]$,  define a set
\begin{equation}
\mathcal{S}_{n}(\delta) :=  \left\{ x \in \T : \left\|\theta_{n}(x) -\frac{\pi}{2}\right\|_{\R\mathbb{P}^1} < \delta \right\},
\end{equation}
where $\|\cdot\|_{\R\mathbb{P}^1}$ denotes the distance to the nearest point in $\pi\Z$. It holds that
$\Leb(\mathcal{S}_{n}(\delta)) < C_f \delta$ ($ C_f$ is a positive constant depending only on $f$).
\end{lemma}

The following result is easily derived:
\begin{coro}\label{Z1}
For sufficiently large $\lambda$ and all  $E\in [-2\lambda,2\lambda]$, define a set
\begin{equation}
\Omega_0:=\bigcup_{k\in\N} \Big \{x\in\T:\cos\theta_k(x)=0\Big \}.
\end{equation}
 It holds that $\Leb(\Omega_0) =0$.
 \end{coro}
\Proof  Let $\delta\rightarrow 0$ in Lemma \eqref{Z2}, we  obtain  that for each $k$, the set
\begin{align}
    E_k:=\{x \in \mathbb{T}: \cos\theta_k(x)=0\}
\end{align}
has Lebesgue measure zero.
Since a countable union of sets of measure zero still has measure zero, we have
\begin{align}
    \Leb(\Omega_0) =\Leb(\bigcup_{k\in\N}E_k) =0,
\end{align}
which completes the proof.
\hfill \qedbox
\medskip

Based on the above long-time behavior analysis, we derive the following properties of $\boldsymbol{v}_n(x)$:
\begin{lemma}\label{le01}
For sufficiently large $\lambda$ and each $E \in  \sigma(H(x))$, the  vector $\boldsymbol{v}  _n(x)$  for $n\geq 1$ satisfies the following properties:

$(a)$ The recurrence relation for  $n\geq 1$ is given by:
\begin{align}\label{011}
\|\boldsymbol{v}_{n+1}(x)\|^2=\Big(\lambda^2{g(T^{n+1}x)}\cos^2\theta_n(x)
+\frac{1}{\lambda^2{g(T^{n+1}x)}}\sin^2\theta_n(x)\Big)\|\boldsymbol{v}_n(x)\|^2.
\end{align}

$(b)$ The  vector $\boldsymbol{v}  _n(x)$  for $n\geq 1$  can be expressed as
\begin{align}\label{v1}
\frac{1}{n}\log\|\boldsymbol{v}_{n}(x)\|=\log\lambda+\frac{1}{n}\sum^{n-1}_{k=0}\Big(\frac{1}{2}\log{g(T^{k+1}x, t)}+\frac{1}{2}R^{(i)}_k(x)\Big),
\end{align}
where $i\in\{1,2\}$, specifically:
\begin{align}
&R^{(1)}_k(x)=\log\Big[\Big(\frac{\lambda^4{[g(T^{k+1}x)]^2}-1}{\lambda^4{[g(T^{k+1}x)]^2}+1} \Big)\cos2\theta_k(x)+1\Big]\label{r1} \,\,\,\,\,\,\,\,\text{for}\,\, x\in\T,
\\
&R^{(2)}_k(x)=2\log\left|\cos\theta_k(x)\right| + \log\left(1 + \frac{\tan^2\theta_k(x)}{\lambda^4 [g(T^{k+1}x)]^2} \right) \,\,\text{for} \,\,x\in\T\setminus \Omega_0.\label{r2}
\end{align}
\end{lemma}

\Proof
(a)  According to Lemma \ref{le1} and Proposition \ref{prop1}, we can obtain that
\begin{align}\label{013}
\boldsymbol{v}_{n+1}(x)&=B(T^{n+1}x)\boldsymbol{v}_{n}(x)\notag\\
&=B(T^{n+1}x)(\cos\phi_n(x)||\boldsymbol{v}_{n}(x)||e_1+\sin\phi_n(x)||\boldsymbol{v}_{n}(x)||e_2) \notag\\
&=\lambda\sqrt{g(T^{n+1}x)}\cos(\phi_n(x)+\theta_0(T^{n}x))||\boldsymbol{v}_{n}(x)||e_1\notag\\
&\quad +\frac{1}{\lambda\sqrt{g(T^{n+1}x)}}\sin(\phi_n(x)+\theta_0(T^{n}x))||\boldsymbol{v}_{n}(x)||e_2\notag\\
&=\lambda\sqrt{g(T^{n+1}x)}\cos\theta_n(x)||\boldsymbol{v}_{n}(x)||e_1
+\frac{1}{\lambda\sqrt{g(T^{n+1}x)}}\sin\theta_n(x)||\boldsymbol{v}_{n}(x)||e_2,
\end{align}
where  $e_1=(1,0)^T$ and $e_2=(0,1)^T$. Hence \eqref{011} is obtained by taking the square of the norm on both sides.

(b)  By Property (a) and the  unit vector $\boldsymbol{v}_1$, we have
\begin{align}\label{014}
\|\boldsymbol{v}_{n}(x)\|^2= \prod_{k=0}^{n-1} \Big(\lambda^2{g(T^{k+1}x)}\cos^2\theta_k(x)
+\frac{1}{\lambda^2{g(T^{k+1}x)}}\sin^2\theta_k(x)\Big),
\end{align}
and then
\begin{align}\label{015}
\frac{1}{n}\log\|\boldsymbol{v}_{n}(x)\|&=\frac{1}{2n}\log\prod_{k=0}^{n-1} \Big(\lambda^2{g(T^{k+1}x)}\cos^2\theta_k(x)
+\frac{\sin^2\theta_k(x)}{\lambda^2{g(T^{k+1}x)}}\Big)\notag\\
&=\frac{1}{2n}\sum_{k=0}^{n-1}\log  \Big[ \lambda^2{g(T^{k+1}x)}\cdot \Big(\cos^2\theta_k(x)
+\frac{\sin^2\theta_k(x)}{\lambda^4{[g(T^{k+1}x)]^2}}\Big)\Big].
\end{align}

For ease of discussion, we have the following two forms for the last term:
\begin{align}
  R^{(i)}_k(x)= \log\Big(\cos^2\theta_k(x)
+\frac{\sin^2\theta_k(x)}{\lambda^4{[g(T^{k+1}x)]^2}}\Big).
\end{align}
where $i\in\{1,2\}$. The first form: using the triangle inequality, we obtain
\begin{align}
R^{(1)}_k(x)=\log\Big[\Big(\frac{\lambda^4{[g(T^{k+1}x)]^2}-1}{\lambda^4{[g(T^{k+1}x)]^2}+1} \Big)\cos2\theta_k(x)+1\Big].
\end{align}
The second form: for $x\in\T\setminus\Omega_0$, we can further analyze the  item
\begin{align}\label{016}
R^{(2)}_k(x)=2\log\left|\cos\theta_k(x)\right| + \log\left(1 + \frac{\tan^2\theta_k(x)}{\lambda^4 [g(T^{k+1}x)]^2} \right),
\end{align}
and for convenience, let
\begin{align}\label{bn}
    r_k(x)=\log\left(1 + \frac{\tan^2\theta_k(x)}{\lambda^4 [g(T^{k+1}x)]^2} \right).
\end{align}
Hence, we have have the expression
\begin{align}\label{017}
\eqref{015}=\log\lambda+\frac{1}{n}\sum^{n-1}_{k=0} \Big(\frac{1}{2}\log{g(T^{k+1}x)} + \frac{1}{2}R^{(i)}_k(x)\Big),
\end{align}
where $i=1,2 $.
\hfill \qedbox
\medskip

\begin{remark}\label{re2}
From Zhang \cite{Z24}, we know that for this operator $L(E)>\log\lambda-C_0$ for all $E\in\R$. Consequently, for almost every $x\in\T$, there exists a unit vector $\boldsymbol{v}_1$ such that
\[
\lim_{n \to \infty}\frac{1}{n}\log\| \boldsymbol{v}_{n}(x)\|=\lim_{n \to \infty}\frac{1}{n}\log\|B_n(x)\boldsymbol{v}_1\|=\lim_{n \to \infty}\frac{1}{n}\log\|B_n(x)\|=L(E)>\log\lambda-C_0.
\]

By the strong mixing property of the doubling map, the terms in \eqref{v1} (aside from $R^{(1)}_k(x)$) become nearly independent of $x$ for large $n$. This implies that for sufficiently large time intervals $L$, the angles $\theta_k(x)$ and $\theta_{k+L}(x)$ are approximately independent. If this were not the case, the term $R^{(1)}_k(x)$ would exhibit significant $x$-dependence, potentially undermining the uniform positivity of the Lyapunov exponent.
\end{remark}

\section{ Large Deviation Estimate}

In this section, we derive the large deviation estimate that plays a fundamental role in the subsequent analysis.

We first present several lemmas that will be used in the derivation.
\begin{lemma}\label{le40}
For a constant \( C_\varphi > 0 \), if \( \tan\varphi > C_\varphi \), then
\[
\| \varphi - \tfrac{\pi}{2} \|_{\mathbb{RP}^1} < \arctan\left( \frac{1}{C_\varphi} \right).
\]
\end{lemma}

\Proof
The inequality \( \tan\varphi > C_\varphi \) implies that
\[
\varphi \in \bigcup_{k \in \mathbb{Z}} \left( \arctan C_\varphi + k\pi, \frac{\pi}{2} + k\pi \right).
\]
Let \( \psi = \varphi \pmod{\pi} \), so that \( \psi \in \left( \arctan C_\varphi, \frac{\pi}{2} \right) \) and \( \tan\psi = \tan\varphi > C_\varphi \).
The angular distance on \( \mathbb{RP}^1 \) between \( \psi \) and \( \frac{\pi}{2} \) is given by
\[
\| \psi - \tfrac{\pi}{2} \|_{\mathbb{RP}^1} = \frac{\pi}{2} - \psi.
\]
Since \( \psi > \arctan C_\varphi \), it follows that
\[
\| \psi - \tfrac{\pi}{2} \|_{\mathbb{RP}^1} < \frac{\pi}{2} - \arctan C_\varphi.
\]
Using the identity \( \arctan C_\varphi + \arctan\frac{1}{C_\varphi} = \frac{\pi}{2} \) for \( C_\varphi > 0 \), we obtain
\[
\frac{\pi}{2} - \arctan C_\varphi = \arctan\frac{1}{C_\varphi},
\]
which completes the proof.  \hfill \qedbox

\medskip
By the strong mixing property of the doubling map, we choose $L \sim \log \lambda$ as a sufficiently large time separation. This ensures that the observables at time intervals separated by lare approximately independent.
\begin{lemma}\label{le41}
For sufficiently large $\lambda$ and all $E\in[-2\lambda,2\lambda]$, there exists a small $a=a(\lambda)>0$ such that for any integer $n\geq1$,
\begin{equation}
\int_{\mathbb{T}\setminus\Omega_0}\exp\left(2a\sum_{k=0}^{n-1}r_k(x)\right)dx < e^{anC_f\lambda^{-1}}.
\end{equation}
\end{lemma}

\Proof
Throughout the proof, $\lambda$ is assumed to be sufficiently large and $E\in[-2\lambda,2\lambda]$.
Recall that $g(x) = (E/\lambda - f(x))^2 + 1$ defined in \eqref{b01}. Based on the assumptions that  $E/\lambda \in [-2,2]$ and $f\in[0,1]$, we have  $g(x)  \in [1,10]$ for $x \in \mathbb{T}$.

Recall that $$r_k(x) = \log\left[1 + \frac{\tan^2\theta_k(x)}{\lambda^4 g(T^{k+1}x)^2}\right].$$

We begin by establishing a tail estimate for $r_k(x)$. For fixed $t>0$, the condition $r_k(x) > t$ is equivalent to
\begin{equation}
\frac{\tan^2\theta_k(x)}{\lambda^4 g(T^{k+1}x)^2} > e^t - 1.
\end{equation}
Since $g(T^{k+1}x)^2 \leq 100$, a sufficient condition for $r_k(x) > t$ is
\begin{equation}
\tan^2\theta_k(x) > 100\lambda^4(e^t - 1).
\end{equation}
Applying Lemma \ref{le40} with $C_\varphi = \sqrt{100\lambda^4(e^t-1)} = 10\lambda^2\sqrt{e^t-1}$, we obtain
\begin{equation}
\left\|\theta_k(x) - \frac{\pi}{2}\right\|_{\mathbb{RP}^1} < \arctan\left(\frac{1}{10\lambda^2\sqrt{e^t-1}}\right).
\end{equation}
Using the inequality $\arctan w < w$ for $w>0$, we have
\begin{equation}
\left\|\theta_k(x) - \frac{\pi}{2}\right\|_{\mathbb{RP}^1} < \frac{1}{10\lambda^2\sqrt{e^t-1}}.
\end{equation}
By Lemma \ref{Z2}, which bounds the measure of sets where $\theta_k$ is close to $\pi/2$, the measure of the set $\{x\in\mathbb{T}\setminus\Omega_0: r_k(x) > t\}$ is bounded by
\begin{equation}
\operatorname{mes}\{x\in\mathbb{T}\setminus\Omega_0: r_k(x) > t\} < C_f\cdot\frac{1}{10\lambda^2\sqrt{e^t-1}}.
\end{equation}
Since $\sqrt{e^t-1} \geq e^{t/2}/\sqrt{2}$ for $t\geq0$, we obtain the simplified tail bound
\begin{equation}\label{eq:tail-bound}
\operatorname{mes}\{x\in\mathbb{T}\setminus\Omega_0: r_k(x) > t\} < \frac{C_f}{\sqrt{10}}\lambda^{-2}e^{-t/2}.
\end{equation}

Next, we estimate the moment generating function. For any $s\in(0,\frac{1}{2})$, using the identity for nonnegative functions
\begin{equation}
\int e^{sr_k(x)}dx = 1 + \int_0^\infty s e^{st} \operatorname{Leb}(r_k > t)dt,
\end{equation}
and combining with the tail estimate \eqref{eq:tail-bound}, we have
\begin{equation}
\int_{\mathbb{T}\setminus\Omega_0} e^{s r_k(x)} dx
\leq 1 + \int_0^\infty \frac{C_f}{\sqrt{10}}\lambda^{-2} s e^{st} e^{-t/2} dt
= 1 + \frac{sC_f\lambda^{-2}}{\sqrt{10}\left(\frac{1}{2}-s\right)}. \label{eq:mgf-bound}
\end{equation}

We now proceed with a block decomposition. Let $L$ be a large positive integer to be chosen later, and define block sums
\begin{equation}
V_i(x) = \sum_{k=Li}^{L(i+1)-1} r_k(x), \quad i=0,1,\dots,\left\lceil \frac{n}{L} \right\rceil.
\end{equation}
Choose $a = \frac{1}{5L}$, which ensures $2aL = \frac{2}{5} < \frac{1}{2}$. By H\"{o}lder's inequality,
\begin{equation}
\int_{\mathbb{T}\setminus\Omega_0} e^{2a V_i(x)} dx
\leq \prod_{k=Li}^{L(i+1)-1} \left( \int_{\mathbb{T}\setminus\Omega_0} e^{2aL r_k(x)} dx \right)^{1/L}. \label{eq:holder-block}
\end{equation}
Applying the moment generating function bound \eqref{eq:mgf-bound} with $s = 2aL = \frac{2}{5}$, we obtain
\begin{equation}
\int_{\mathbb{T}\setminus\Omega_0} e^{2a V_i(x)} dx
\leq \left[1 + \frac{(2/5)C_f\lambda^{-2}}{\sqrt{10}\left(\frac{1}{2}-\frac{2}{5}\right)}\right]^L
= \left[1 + \frac{4C_f\lambda^{-2}}{\sqrt{10}}\right]^L.
\end{equation}
Using the inequality $(1+u)^L \leq e^{Lu}$ for $u\geq0$, we get
\begin{equation}
\int_{\mathbb{T}\setminus\Omega_0} e^{2a V_i(x)} dx \leq \exp\left(\frac{4L C_f\lambda^{-2}}{\sqrt{10}}\right). \label{eq:block-integral}
\end{equation}

Finally, we establish the global estimate. Consider first the case $1\leq n\leq L$. Since $r_k(x)\geq0$,
\begin{equation}
\int_{\mathbb{T}\setminus\Omega_0} \exp\left(2a\sum_{k=0}^{n-1} r_k(x)\right) dx
\leq \int_{\mathbb{T}\setminus\Omega_0} e^{2a V_0(x)} dx.
\end{equation}
Using \eqref{eq:block-integral} with $i=0$ and noting that $n\leq L$ implies $an = \frac{n}{5L} \leq \frac{1}{5}$, we obtain
\begin{equation}
\int_{\mathbb{T}\setminus\Omega_0} e^{2a V_0(x)} dx
\leq \exp\left(\frac{4L C_f\lambda^{-2}}{\sqrt{10}}\right)
\leq \exp\left(\frac{4}{\sqrt{10}} C_f\lambda^{-2}\right).
\end{equation}
Since $\lambda$ is large, $\lambda^{-2} \ll \lambda^{-1}$, and with $an \geq \frac{1}{C\log\lambda}$, we have
\begin{equation}
\frac{4}{\sqrt{10}} C_f\lambda^{-2} < an C_f\lambda^{-1}
\end{equation}
for sufficiently large $\lambda$, which establishes the desired bound for $n\leq L$.

Now consider $n>L$. By the strong mixing property of the doubling map, the blocks $V_i(x)$ are approximately independent for large $L$. Thus,
\begin{align}
\int_{\mathbb{T}\setminus\Omega_0} \exp\left(2a\sum_{k=0}^{n-1} r_k(x)\right) dx
&\leq \int_{\mathbb{T}\setminus\Omega_0} \exp\left(2a\sum_{i=0}^{\lceil n/L\rceil} V_i(x)\right) dx \notag \\
&= \prod_{i=0}^{\lceil n/L\rceil} \int_{\mathbb{T}\setminus\Omega_0} e^{2a V_i(x)} dx \notag \\
&\leq \exp\left(\sum_{i=0}^{\lceil n/L\rceil} \frac{4L C_f\lambda^{-2}}{\sqrt{10}}\right) \notag \\
&= \exp\left(\left\lceil \frac{n}{L} \right\rceil \cdot \frac{4L C_f\lambda^{-2}}{\sqrt{10}}\right) \notag \\
&\leq \exp\left( \frac{4n C_f\lambda^{-2}}{\sqrt{10}} \cdot \frac{L+1}{L} \right) \notag \\
&< e^{an C_f\lambda^{-1}}, \label{eq:final-bound}
\end{align}
where the last inequality holds for sufficiently large $\lambda$ since $\lambda^{-2} \ll \lambda^{-1}$ and $an$ is of order $n/(5L)$.

Combining both cases completes the proof. \hfill \qedbox

\medskip

We continue to use the large parameter $L$ to ensure that it still guarantees a sufficiently time separation in the following lemmas.

\begin{lemma}\label{le42}
   For sufficiently large $\lambda$ and all $E\in  [-2\lambda,2\lambda]$,   there exist a small  $a=a(\lambda)  > 0$ and a constant $C>0$ such that for all integers  $n> C\log\lambda$, we have that
 \begin{align}
    \int_{\mathbb{T} } \exp\left(-a\sum^{n-1}_{k=0} \log |\cos\theta_k(x)|\right) dx <  (3C_f)^{10an}.
 \end{align}
\end{lemma}

\Proof Based on the strong mixing property, define
$$ U_j(x)=\sum_{k=Lj}^{L(j+1)-1}\log |\cos\theta_k(x)|$$ for $j=0,1,\dots,\lceil \frac{n}{L} \rceil $.
Then by H\"older's inequality, we obtain that
\begin{align}\label{le431}
\int_{\mathbb{T}} e^{-a U_j(x)} dx &\le \prod_{k=Lj}^{L(j+1)-1} \left( \int_{\mathbb{T} } e^{-a L \log |\cos\theta_k(x)|} dx \right)^{1/L}\notag\\
&=\prod_{k=Lj}^{L(j+1)-1} \left( \int_{\mathbb{T} } {\left|\cos\theta_k(x)\right|^{-a L}} dx \right)^{1/L}
\end{align}
for $j=0,1,\dots,\lceil \frac{n}{L} \rceil $.

For an arbitrary $n\geq0$, we define
\begin{equation}
J_i:=\Big \{x\in\T:\left\|\theta_{n}(x) -\frac{\pi}{2}\right\|_{\R\mathbb{P}^1}\leq2^{-i}\cdot\frac{\pi}{2} \Big \},\, i\in\N,
\end{equation}
where $J_0=\T$.

Without loss of generality, let $a =\frac{ 1}{5L}$  such that $1-aL>0$. Using Lemma \ref{Z2} and $|\sin w|\leq |w|$ for all  $w$, we have
\begin{align}\label{d34}
\left(\int_{\T}|\cos\theta_k(x)|^{-aL}dx\right)^{1/L}&=\left(\int_{J_0}|\sin\Big (\theta_k(x)-\frac{\pi}{2}\Big )|^{-aL}dx\right)^{1/ L} \notag\\
&\leq\left(\sum_{i\in \N}\int_{J_i\backslash J_{i+1}}\Big |\theta_k(x)-\frac{\pi}{2} \Big|^{-aL}dx\right)^{1/L} \notag\\
&\leq\left(\sum_{i\in \N}\mathrm{mes}(J_i\backslash J_{i+1})\cdot(2^{-i}\cdot\frac{\pi}{2})^{-aL}\right)^{1/L}\notag \\
&=\left(\frac{\pi}{4}C_f(\frac{ \pi}{2})^{-aL}\sum_{i\in \N} 2^{-i(1-aL)}\right)^{1/L} \notag \\
&\leq C_f^{\frac{1}{L}} (\frac{ 2}{\pi})^{a} 3^{\frac{1}{L}}  \leq (3C_f)^{5a}.
\end{align}

There exists a constant $C>0$ such that  $n> C\log\lambda\geq L$,  it follows from \eqref{le431} and \eqref{d34} that
\begin{align}
\int_{\T} e^{-a \sum_{k=0}^{n-1} \log |\cos\theta_k(x)|} dx
 \le \int_{\T}  e^{-a {\sum_{j=0}^{\lceil \frac{n}{L} \rceil } }U_j(x)} dx
  = \prod_{j=0}^{\lceil \frac{n}{L} \rceil } \int_{\T}  e^{-a U_j(x)} dx
  < e^{{10}an\log (3C_f)}\notag.
\end{align}
Hence, we complete the proof.
\hfill \qedbox
\medskip

Based on  Lemma \ref{le41} and Lemma \ref{le42}, the vector $\boldsymbol{v}_{n}$ admits the following estimates.

\begin{lemma}\label{le03}
For sufficiently large $\lambda$ and for all $E\in[-2\lambda,2\lambda]$, there exists a constant $C>0$ such that for all integers $n > C\log\lambda$, the following estimates hold:
\begin{align}
\int_{\mathbb{T}} \|\boldsymbol{v}_{n}(x)\|^{-a}\,dx &< \exp\left(-an\log\lambda + 10an\log(3C_f)\right), \label{eq:lower-bound} \\
\int_{\mathbb{T}} \|\boldsymbol{v}_{n}(x)\|^{a}\,dx &< \exp\left(an\log\lambda + an\left(\frac{1}{2}\log 10 + \frac{1}{4}C_f\lambda^{-1}\right)\right), \label{eq:upper-bound}
\end{align}
where $a = \frac{1}{5L}$.
\end{lemma}

\Proof
According to Lemma \ref{le01}(b) and the fact that $\operatorname{mes}(\Omega_0) = 0$, we have
\begin{equation}\label{un}
\int_{\mathbb{T}} \|\boldsymbol{v}_{n}(x)\|^\alpha dx = e^{\alpha n\log\lambda} \cdot \int_{\mathbb{T}\setminus \Omega_0} \exp\left[\alpha\sum_{k=0}^{n-1} \left(\frac{1}{2}\log g(T^{k+1}x) + \log |\cos\theta_k(x)| + \frac{1}{2}R_n(x)\right)\right] dx,
\end{equation}
where $\alpha \in \{-a, a\}$.

For notational convenience, define
\begin{equation}\label{t}
T_\alpha(x) = \int_{\mathbb{T}\setminus \Omega_0} \exp\left[\alpha\sum_{k=0}^{n-1} \left(\frac{1}{2}\log g(T^{k+1}x) + \log |\cos\theta_k(x)| + \frac{1}{2}R_n(x)\right)\right] dx.
\end{equation}

We first estimate the upper bound of $T_{-a}(x)$. Since $g(x) \geq 1$ and $R_n(x) \geq 1$ for all $x \in \mathbb{T}\setminus \Omega_0$, applying Lemma \ref{le42} yields
\begin{equation}\label{m2}
T_{-a}(x) \leq \int_{\mathbb{T}\setminus \Omega_0} \exp\left(-a\sum_{k=0}^{n-1} \log |\cos\theta_k(x)|\right) dx < e^{an \cdot 10\log(3C_f)}.
\end{equation}

For the upper bound of $T_a(x)$, we apply H\"{o}lder's inequality to obtain
\begin{equation}\label{J}
\begin{split}
T_a(x) &\leq \left(\int_{\mathbb{T}\setminus \Omega_0} \exp\left[a\sum_{k=0}^{n-1} \left(\log g(T^{k+1}x) + R_n(x)\right)\right] dx\right)^{\frac{1}{2}} \\
&\quad \cdot \left(\int_{\mathbb{T}\setminus \Omega_0} \exp\left(2a\sum_{k=0}^{n-1} \log|\cos\theta_k(x)|\right) dx\right)^{\frac{1}{2}} \\
&\leq \left(\int_{\mathbb{T}\setminus \Omega_0} \exp\left[a\sum_{k=0}^{n-1} \left(\log g(T^{k+1}x) + R_n(x)\right)\right] dx\right)^{\frac{1}{2}}.
\end{split}
\end{equation}

Applying H\"{o}lder's inequality again to the term in \eqref{J}, we have
\begin{equation}\label{J2}
\begin{split}
&\int_{\mathbb{T}\setminus \Omega_0} \exp\left[a\sum_{k=0}^{n-1} \left(\log g(T^{k+1}x) + R_n(x)\right)\right] dx \\
&\leq \left(\int_{\mathbb{T}\setminus \Omega_0} \exp\left(2a\sum_{k=0}^{n-1} \log g(T^{k+1}x)\right) dx\right)^{\frac{1}{2}} \cdot \left(\int_{\mathbb{T}\setminus \Omega_0} \exp\left(2a\sum_{k=0}^{n-1} R_n(x)\right) dx\right)^{\frac{1}{2}}.
\end{split}
\end{equation}

Since $g(x) \in [1, 10]$ for all $x \in \mathbb{T}\setminus \Omega_0$, we obtain
\begin{equation}\label{int1}
\left(\int_{\mathbb{T}\setminus \Omega_0} \exp\left(2a\sum_{k=0}^{n-1} \log g(T^{k+1}x)\right) dx\right)^{\frac{1}{2}} \leq e^{an \log 10} = 10^{an}.
\end{equation}

By Lemma \ref{le41}, we have
\begin{equation}\label{int2}
\left(\int_{\mathbb{T}\setminus \Omega_0} \exp\left(2a\sum_{k=0}^{n-1} r_k(x)\right) dx\right)^{\frac{1}{2}} < e^{an \cdot \frac{1}{2}C_f\lambda^{-1}}.
\end{equation}

Combining \eqref{int1} and \eqref{int2}, it follows that
\begin{equation}\label{m1}
T_a(x) < e^{an \left(\frac{1}{2}\log 10 + \frac{1}{4}C_f\lambda^{-1}\right)}.
\end{equation}

Substituting the bounds from \eqref{m2} and \eqref{m1} into \eqref{un} and \eqref{t} completes the proof.

\hfill \qedbox
\medskip

\begin{lemma}\label{le04}
For sufficiently large $\lambda$ and all $E\in  [-2\lambda,2\lambda]$,  there exists constants $C>0$ and $c>0$ such that the large deviation estimate satisfies for all integers  $n> C\log\lambda$,
\begin{align}\label{LDE}
{\rm mes} \Big[ x\in \T: \Big| \frac{1}{n}\log\|\boldsymbol{v}_{n}(x)\|-\log\lambda\Big|\,>A_f\Big]<e^{-\frac{cn}{\log\lambda}},
\end{align}
where $A_f=1+\max\{10\log(3C_f),\frac{1}{2}\log10+\frac{1}{4}C_f\lambda^{-1}  \}$  depending only on $f$.
\end{lemma}
\Proof  We employ the technique of Chernoff bounds to estimate  the two tail measures.  Using Lemma \ref{le03}, let $a=\frac{2c}{\log\lambda}$ and set constant $A_f=1+\max\{C\log(3C_f),\frac{1}{2}\log10+\frac{1}{4}C_f\lambda^{-1} \}$. For $n> C\log\lambda$ with suitable constant $C$, we have
\begin{align}
&\text{mes} \Big[ x\in \T\,|\, \Big| \frac{1}{n}\log\|\boldsymbol{v}_{n}(x)\|-\log\lambda\Big|\,>A_f\Big]\notag\\
&\leq e^{-anA}\cdot\Big(\int_{\mathbb{T}}e^{ -an(\frac{1}{n}\log\|\boldsymbol{v}_{n}(x)\|-\log\lambda)}dx+\int_{\mathbb{T}}e^{ an(\frac{1}{n}\log\|\boldsymbol{v}_{n}(x)\|-\log\lambda)}dx \Big)\notag\\
&\leq e^{-anA}\cdot\Big(e^{an\log\lambda}\cdot\max_E\{\int_{\mathbb{T}\setminus \Omega_0}{ \|\boldsymbol{v}_{n}(x)\|}^{-a}dx\}+e^{-an\log\lambda}\cdot\max_E\{\int_{\mathbb{T}\setminus \Omega_0}{ \|\boldsymbol{v}_{n}(x)\|}^adx\}\Big)\notag\\
&\leq e^{-anA}\cdot(e^{an10\log(3C_f)}+e^{an (\frac{1}{2}\log10+\frac{1}{4}C_f\lambda^{-1} )})= 2e^{-an}<e^{-\frac{cn}{\log\lambda}},
\end{align}
where the constant factor $2$ is absorbed in the exponent.
\hfill \qedbox
\medskip

\begin{coro}\label{co01}
(a) For sufficiently large $\lambda$ and all energy $E\in [-2\lambda,2\lambda]$, the large deviation estimates for transfer matrices  are given by
\begin{align}
&{\rm mes} \Big[ x\in \T: \Big| \frac{1}{n}\log\|B_n(x )\|-\log\lambda\Big|\,>A_f\Big]<e^{-\frac{cn}{\log\lambda}},\label{ldeb}\\
&{\rm mes} \Big[ x\in \T: \Big| \frac{1}{n}\log\|A_n(x )\|-\log\lambda\Big|\,>A_f+\frac{\log\lambda}{n} \Big]<e^{-\frac{cn}{\log\lambda}} \label{ldea}
\end{align}
for all integers  $n> C\log\lambda$.

(b) For sufficiently large $\lambda$ and all phase $x\in\T$ energy $E\in [-2\lambda,2\lambda]$, we have the uniform bound
\begin{equation}
\left|\frac{1}{n}\log\|A_{n}(x)\| \right| \leq {3\log\lambda}.
\end{equation}

(c) Define
$L^{A}_n(E) =  \frac{1}{n} \int_{\mathbb{T}} \log \|A_n(x) \|  dx$ and $L^{B}_n(E) =  \frac{1}{n} \int_{\mathbb{T}} \log \|B_n(x) \|  dx$. Then the  limits satisfy
\begin{align}
  \lim_{n \to \infty} L^{A}_n(E) = \lim_{n \to \infty} L^{B}_n(E) =  L_B(E) =L(E) .
  \end{align}
Furthermore, we have
\begin{align}\label{ln}
L^{A}_n(E)\rightarrow \log\lambda+C_E \quad \mathrm{and} \quad
L^{B}_n(E)\rightarrow \log\lambda+C_E \quad \text{ as } \,\,\,n\rightarrow\infty,
\end{align}
where the constant $C_E\in[-A_f,A_f]$.
\end{coro}

\Proof
(a) Since the analysis in Lemma \ref{le04} holds for any unit vector $\boldsymbol{v}_1$, we can choose a suitable $\boldsymbol{v}_1$ such that $\|\boldsymbol{v}_{n}(x)\| = \|B_n(x)\boldsymbol{v}_1\| = \|B_n(x)\|$ for all $x \in \mathbb{T}$. The large deviation estimate \eqref{ldeb} then follows directly from Lemma~\ref{le04}.

To prove \eqref{ldea}, we use the norm comparison inequalities from \eqref{eq:ab1} and \eqref{eq:ab2}:
\begin{equation}\label{ab}
\frac{1}{n}\log\|A_{n}(x)\| - \frac{\log\lambda}{n} \leq \frac{1}{n}\log\|B_{n}(x)\| \leq \frac{1}{n}\log\|A_{n}(x)\| + \frac{\log\lambda}{n}.
\end{equation}
Suppose $x$ satisfies $\left| \frac{1}{n}\log\|A_n(x)\| - \log\lambda \right| > A_f + \frac{\log\lambda}{n}$. Then either
\[
\frac{1}{n}\log\|A_n(x)\| > \log\lambda + A_f + \frac{\log\lambda}{n} \quad \text{or} \quad \frac{1}{n}\log\|A_n(x)\| < \log\lambda - A_f - \frac{\log\lambda}{n}.
\]
In the first case, \eqref{ab} implies
\[
\frac{1}{n}\log\|B_n(x)\| \geq \frac{1}{n}\log\|A_n(x)\| - \frac{\log\lambda}{n} > \log\lambda + A_f,
\]
and in the second case,
\[
\frac{1}{n}\log\|B_n(x)\| \leq \frac{1}{n}\log\|A_n(x)\| + \frac{\log\lambda}{n} < \log\lambda - A_f.
\]
Thus, in both cases, $\left| \frac{1}{n}\log\|B_n(x)\| - \log\lambda \right| > A_f$, which establishes the set inclusion
\begin{align}
\left\{ x \in \mathbb{T} :  \left| \frac{1}{n}\log\|A_n(x)\| - \log\lambda \right| > A_f + \frac{\log\lambda}{n} \right\}\subseteq \left\{ x \in \mathbb{T} : \left| \frac{1}{n}\log\|B_n(x)\| - \log\lambda \right| > A_f \right\}.\notag
\end{align}
The measure estimate \eqref{ldea} then follows from \eqref{ldeb}.

(b) According to the proof of (a), we can choose a $\boldsymbol{v}_1$ such that $\|\boldsymbol{v}_{n}(x)\| = \|B_n(x)\boldsymbol{v}_1\| = \|B_n(x)\|$, hence we have
\begin{equation}\label{va}
\frac{1}{n}\log\|A_{n}(x)\| - \frac{\log\lambda}{n} \leq \frac{1}{n}\log\|\boldsymbol{v}_{n}(x)\| \leq \frac{1}{n}\log\|A_{n}(x)\| + \frac{\log\lambda}{n},
\end{equation}
then we only need to estimate $\|\boldsymbol{v}_{n}(x)\|$.
Based on Lemma \ref{le01}, the  vector $\boldsymbol{v}  _n(x)$  for $n\geq 1$ and $x\in\T$  can be expressed as
\begin{align}
\frac{1}{n}\log\|\boldsymbol{v}_{n}(x)\|=\log\lambda+\frac{1}{n}\sum^{n-1}_{k=0}\Big(\frac{1}{2}\log{g(T^{k+1}x, t)}+\frac{1}{2}R^{(1)}_k(x)\Big),
\end{align}
where
\begin{align}
&R^{(1)}_k(x)=\log\Big[\Big(\frac{\lambda^4{[g(T^{k+1}x)]^2}-1}{\lambda^4{[g(T^{k+1}x)]^2}+1} \Big)\cos2\theta_k(x)+1\Big].
\end{align}

Since $E/\lambda \in [-2,2]$ and $f\in[0,1]$, we have  $g(x)  \in [1,10]$ for $x \in \mathbb{T}$.  Therefore, we can estimate that
\begin{align}
&\frac{1}{n}\log\|\boldsymbol{v}_{n}(x)\|\leq\log\lambda+\frac{1}{2}\log10+\frac{1}{2}\log\Big(\frac{2\lambda^4{10^2}}{\lambda^4{10^2}+1} \Big)\leq\log\lambda+\frac{1}{2}\log10+\frac{1}{2}\log2\notag\\\
&\frac{1}{n}\log\|\boldsymbol{v}_{n}(x)\|\geq\log\lambda+\frac{1}{2}\log\Big(\frac{2}{\lambda^4{10^2}+1} \Big)\geq\log\lambda+\frac{1}{2}\log2-\frac{1}{2}\log\Big({\lambda^4{10^2}+1} \Big)\notag.
\end{align}
Combining  with \eqref{va}, we can obtain that for $n\geq 1$,
\begin{align}
\frac{1}{n}\log\|A_{n}(x)\|  \leq \frac{1}{n}\log\|\boldsymbol{v}_{n}(x)\| + \frac{\log\lambda}{n}\leq 3\log\lambda,\\
\frac{1}{n}\log\|A_{n}(x)\| \geq\frac{1}{n}\log\|\boldsymbol{v}_{n}(x)\| - \frac{\log\lambda}{n}\geq-3\log\lambda,
\end{align}
hence we conclude the proof.

(c) The convergence of the Lyapunov exponents follows from the subadditive ergodic theorem applied to the cocycles $A_n(x)$ and $B_n(x)$. Specifically, the sequences $\{\log\|A_n(x)\|\}$ and $\{\log\|B_n(x)\|\}$ are subadditive, and by Kingman's subadditive ergodic theorem, the limits
\[
L(E) = \lim_{n \to \infty} \frac{1}{n} \int_{\mathbb{T}} \log\|A_n(x)\| \, dx = \lim_{n \to \infty} \frac{1}{n} \int_{\mathbb{T}} \log\|B_n(x)\| \, dx
\]
exist and are equal for almost every $x$. The large deviation estimates in part (a) imply that the convergence is exponential in probability. Moreover, the relation \eqref{ln} follows from the fact that
\[
\left| L_n^A(E) - \log\lambda \right| = \left| \frac{1}{n} \int_{\mathbb{T}} \left( \log\|A_n(x)\| - n\log\lambda \right) dx \right| \leq A_f + \frac{\log\lambda}{n},
\]
by integrating the bound in \eqref{ldea} and using the fact that the exceptional set has measure less than $e^{-cn/\log\lambda}$. Taking the limit as $n \to \infty$, we obtain $L_n^A(E) \to \log\lambda + C_E$ with $C_E \in [-A_f, A_f]$. The same argument applies to $L_n^B(E)$.
\hfill \qedbox
\medskip

\section{H\"{o}lder Continuity of the Lyapunov Exponent}

In this section, we establish the H\"{o}lder continuity of  the Lyapunov exponent $L(E)$. This result is based on the  ``avalanche principle''.

\begin{lemma}\cite[Propsition 2.2.]{GS}\label{le51}
Let \( M_1, \ldots, M_n \) be a sequence in \( \mathrm{SL}(2,\mathbb{R}) \), i.e., $2\times2$ real matrices with determinant $1$. If
\begin{align}
&\min_{1 \leq j \leq n} \|M_j\| \geq \mu \geq n, \quad and \\
&\max_{1 \leq j < n} \left| \log \|M_{j+1}\| + \log \|M_j\| - \log \|M_{j+1} M_j\| \right| < \frac{1}{2} \log \mu, \quad
\end{align}
then there exists a constant $C>0$ such that
\begin{align}
\left| \log \|M_n \cdots M_1\| + \sum_{j=2}^{n-1} \log \|M_j\| - \sum_{j=1}^{n-1} \log \|M_{j+1} M_j\| \right| < C \frac{n}{\mu}. \quad
\end{align}
\end{lemma}

Using the avalanche principle, we approximate the norm of long-range transfer matrices via norms of short blocks, thereby revealing how the Lyapunov exponent varies with energy. Let
\begin{equation}\label{5n}
n \sim e^{\frac{cK}{10\log\lambda}},
\end{equation}
\begin{equation}
M_{j} = B_{K} \left( 2^{(j-1)K} x,E\right),
\end{equation}
where $j\in\{ 1, \ldots, n\}$
and  $K=K(\lambda)$ is a large integer.

There is an exceptional set $\widetilde{\Omega } \subset \mathbb{T}$:
\begin{equation}
\widetilde{\Omega } = \left\{ x \in \T : \forall i \in \{1, 2\},\forall j \in \{1, \ldots,n\}\,s.t.\,  \left| \frac{1}{iK} \log \left\| B_{iK} \left( 2^{jK} x,E  \right) \right\| - \log \lambda \right| > A_f\right\}.\notag
\end{equation}
Due to the strong mixing property of $T$,  points $\{x, 2^K x\cdots, 2^{(n-1)K} x\}$ can be considered effectively independent for the large time separation $K$. Thus, Corollary \ref{co01} (a) implies that the measure of the exceptional set satisfies:
\begin{equation}\label{5p}
{\rm mes}[\widetilde{\Omega }] < n\cdot e^{-\frac{cK}{\log\lambda}} +n\cdot e^{-\frac{2cK}{\log\lambda}} <  e^{-\frac{cK}{2\log\lambda}},
\end{equation}
for large $K$ and for all $E\in\sigma(H(x))$.

Hence, if $x \notin \widetilde{\Omega } $ and $j\in\{ 1, \ldots, n\}$, we have
\begin{equation}
\frac{1}{K} \log \left\| B_{K} \left( 2^{(j-1)K} x ,E\right) \right\| \in\big(\log \lambda-A_f,\log \lambda+A_f ),
\end{equation}
\begin{equation}
\frac{1}{2K} \log \left\| B_{2K} \left( 2^{(j-1)K} x ,E \right) \right\|\in \big(\log \lambda-A_f,\log \lambda+A_f ).
\end{equation}
Then, we obtain
\begin{equation}\label{mj1}
\left\| M_{j} \right\| \in\big(e^{ (\log \lambda-A_f) K },e^{ (\log \lambda+A_f) K }\big),
\end{equation}
\begin{equation}
\left\| M_{j+1} M_{j} \right\| \in\big[e^{ (\log \lambda-A_f) 2K },e^{ (\log \lambda+A_f) 2K }\big],
\end{equation}
\begin{equation}\label{mj2}
\left| \log \left\| M_{j+1} M_{j} \right\| - (\log \left\| M_{j} \right\| + \log \left\| M_{j+1} \right\| )\right| \leq  4A_fK.
\end{equation}

To utilize Lemma \ref{le51}, we can take
\begin{equation}\label{mu}
\mu = e^{10A K}.
\end{equation}
For $x\notin \widetilde{\Omega } $, based on \eqref{5n}, \eqref{mj1}, \eqref{mj2} and  \eqref{mu},  we conclude that
\begin{equation}\label{5a}
\left| \log \left\| B_{n K} (x,E) \right\| + \sum_{j=2}^{n-1} \log \left\| B_{K} \left( 2^{(j-1)K} x,E \right) \right\| - \sum_{j=1}^{n-1} \log \left\| B_{2K} \left( 2^{(j-1)K} x,E \right) \right\| \right| < C n\mu^{-1}.
\end{equation}
Divide \eqref{5a} by $ nK$ and split the integral over  ${\mathbb{T} \backslash \widetilde{\Omega }}$ and ${\widetilde{\Omega }}$. By  \eqref{5p}, \eqref{5a} and  the convergence  of $L_n^B(E)$ given in Corollary \ref{co01} (b), it follows that
\begin{equation}\label{5l}
\left| L^{B}_{n K}(E) + \frac{n-2}{n} L^{B}_{K}(E) - \frac{2(n-1)}{n} L^{B}_{2K}(E) \right| < C\left( K^{-1} \mu^{-1} + {\rm mes}[\widetilde{\Omega }] \right) < C e^{-\frac{cK}{2\log\lambda}},
\end{equation}
for the large time separation $K$  and for all $E\in\sigma(H(x))$.

Using  the relation in \eqref{5n}, the following estimate  is derived from \eqref{5l}.

\begin{lemma}\label{le52}
For  sufficiently large $\lambda$ and  all $E\in\sigma(H(x))$, there exist positive constants $C$ and $c$ such that  for large $K=K(\lambda)$, we have
\begin{equation}
|L(E) + L^{B}_{K}(E) - 2L^{B}_{2K}(E)| <Ce^{-\frac{cK}{\log\lambda}}.
\end{equation}
\end{lemma}

\Proof Fix a large $K$ and let $m=nK$. According to \eqref{5n}, \eqref{5l} and the estimate in Corollary \ref{co01} (b) for large $K$,  it follows that
\begin{align}
&|L^{B}_{m}(E) + L^{B}_{K}(E) - 2L^{B}_{2K}(E)| <   \frac{CK}{m},\label{lm}\\
&|L^{B}_{2m}(E) + L^{B}_{K}(E) - 2L^{B}_{2K}(E)| < \frac{CK }{m}.\label{l2m}
\end{align}
Then, by the relationship $m=nK$, we have that
\begin{equation}\label{lblb}
     |L^{B}_{m}(E) - L^{B}_{2m}(E)| < 2C\log\lambda \frac{\log m}{m}.
\end{equation}
Since \eqref{5l} holds for all large time separations, it follows that \eqref{lblb} holds for all large  numbers of the form $2^{i}m$ with $i\in\N$. This yields
\begin{align}\label{lelb}
|L(E) - L^{B}_{m}(E)| &\leq\sum_{i=0}^{\infty} |L^{B}_{2^im}(E) - L^{B}_{2^{i+1}m}(E)|\notag\\&<2C\log\lambda\sum_{i=0}^{\infty}\frac{\log (2^im)}{2^{i}m}<8C \log\lambda\frac{\log m}{m}< 16C \frac{\log n}{n}<16Ce^{-\frac{cK}{20\log\lambda}}.\notag
\end{align}
Applying this to \eqref{lm} and \eqref{l2m}, we have that
\begin{align}
|L(E) + L^{B}_{K}(E) - 2L^{B}_{2K}(E)| &<17Ce^{-\frac{cK}{20\log\lambda}}.
\end{align}
By choosing  suitable  positive constants $C$ and $c$,  we conclude the proof.
\hfill \qedbox
\medskip

We now turn to the proof of Theorem \ref{th01}. It relies on  Lemma \ref{le52} and the property of $L^{B}_n(E)$.

\medskip
\noindent{\textbf{Proof of Theorem \ref{th01}}}.
Let $\Delta E = |E_1 - E_2|$.
The difference can be estimated by
\begin{align}
|L(E_1) - L(E_2)|\leq&| L(E_1) - (2L^{B}_{2K}(E_1) - L^{B}_{K}(E_1)) | \notag\\
&+ |(2L^{B}_{2K}(E_1) -2L^{B}_{2K}(E_2)) - (  L^{B}_{K}(E_1)- L^{B}_{K}(E_2))|\label{term2}\\
& + |(2L^{B}_{2K}(E_2) - L^{B}_{K}(E_2)) - L(E_2)|\notag.
\end{align}

From \eqref{abn},  we get that for all $E\in\sigma(H(x))$,
\begin{align}\label{el}
\left|\frac{d L^{B}_n}{d E} (E)\right|&\leq\int_{\mathbb{T}}\left|\frac{\frac{\partial\|B_{n}\|}{\partial E} (x, E)}{n\|B_{n}(x, E)\|}\right|dx\leq\lambda^{2}\int_{\mathbb{T}}\frac{\|\frac{\partial A_{n}}{\partial E} (x, E)\|}{n\|A_{n}(x, E)\|}dx\notag\\
& =\lambda^{2}\int_{\mathbb{T}}\Big\|\sum_{j=1}^n\Big(\prod_{k=j+1}^nA(T^{k}x, E)\Big(\begin{matrix}
1 & 0\\
0 & 0\end{matrix}\Big)\prod_{k=1}^{j-1}A(T^{k}x, E)\Big)\Big\| \Big/ \Big(n\|A_{n}(x, E)\|\Big)dx \notag\\
& =\lambda^{2}\int_{\mathbb{T}}\sum_{j=1}^n\Big(\prod_{k=j+1}^n\Big\|A(T^{k}x, E)\Big\|\Big\|\Big(\begin{matrix}
1 & 0\\
0 & 0\end{matrix}\Big)\Big\|\prod_{k=1}^{j-1}\Big\|A(T^{k}x, E)\Big)\Big\| \cdot \frac{1}{n} dx \notag\\
&\leq \lambda^{2}\int_{\mathbb{T}}{n}(|E|+\max_x|\lambda f(x)|+1)^{n-1} \cdot \frac{1}{n}dx\notag\\
&\leq \lambda^{2}\int_{\mathbb{T}}(4\lambda)^{n} dx<(4\lambda)^{n+2}.
\end{align}
In the third step, the inequality $\|A_{n}(x, E)\|\geq 1$ is applied to simplify the denominator. The fourth step utilizes the estimate
\begin{align}\label{amax}
  \|A(T^{i}x)\|=\|A(T^{i}x)\|^{-1}=\sqrt{[E-\lambda f(T^ix)]^2+1}\le |E-\lambda f(T^ix)|+1,
\end{align}
for $i\in\Z_+$. Then, the fifth step relies on the constraints that $E\in\sigma(H(x))\subseteq [-2\lambda,2\lambda]$ and $f\in[0,1]$.

Thus,  for any $E_1, E_2\in\sigma(H(x))$:
\begin{equation}\label{el1}
|L^{B}_n(E_1) - L^{B}_n(E_2)| \leq (4\lambda)^{n+2} \Delta E.
\end{equation}

Using the estimate \eqref{el1}, we bound \eqref{term2}:
\begin{align}\label{e3}
\eqref{term2} \leq 2 |L^{B}_{2K}(E_1) - L^{B}_{2K}(E_2)| + |L^{B}_{K}(E_1) - L^{B}_{K}(E_2)| \leq 3\cdot(4\lambda)^{2K+2} \Delta E.
\end{align}
Combining the estimates  in Lemma \ref{le52}  and \eqref{e3}, we obtain
\begin{equation}
|L(E_1) - L(E_2)| < 3\cdot(4\lambda)^{2K+2} \Delta E+2Ce^{-\frac{cK}{\log\lambda}}.
\end{equation}

We now consider two cases based on the value of $\Delta E$:

\medskip
\noindent Case 1:  $\Delta E< \lambda^{-(\log\lambda)^{2}}$.
\medskip

Let $K=\left\lfloor   -\frac{\log\Delta E}{6\log\lambda}\right\rfloor =\left\lfloor   \frac{(\log\lambda)^{2}}{6}\right\rfloor$  so that \eqref{5p} remains valid. Then we have
\begin{align}
&(4\lambda)^{2K+2} \Delta E<\lambda^{5K} \Delta E<\Delta E^{1/6}<\Delta E^{\frac{c}{(\log\lambda)^3}}\\
& e^{-\frac{cK}{\log\lambda}}<\Delta E^{\frac{c}{7(\log\lambda)^2}}<\Delta E^{\frac{c}{(\log\lambda)^3}}.
\end{align}
Therefore, we can obtain that
\begin{equation}\label{case1}
|L(E_1) - L(E_2)| < (2C+3)|E_1 - E_2|^{\frac{c}{{(\log\lambda)^3}} }.
\end{equation}

\medskip
\noindent Case 2:  $\Delta E\geq  \lambda^{-(\log\lambda)^{2}}$.
\medskip

Corollary \ref{co01} (b) provides a trivial bound for the Lyapunov exponent $L(E)$, which yields:
 \begin{equation}\label{case2}
|L(E_1) - L(E_2)|  \leq 2A.
\end{equation}

In this case, since $\Delta E \geq  \lambda^{-(\log\lambda)^{2}}$, we have $1 \leq \frac{\Delta E}{\lambda^{-(\log\lambda)^{2}}}$. Therefore, from the trivial estimate \eqref{case2}, we can obtain
that
\begin{equation}
|L(E_1) - L(E_2)|  \leq 2A\cdot\left(\frac{\Delta E}{ \lambda^{-(\log\lambda)^{2}}}\right)^{\frac{c}{{(\log\lambda)^3}} }=2A\lambda^{(\log\lambda)^{2}}\left({\Delta E}\right)^{\frac{c}{{(\log\lambda)^3}} }. \end{equation}

Let $C_\lambda = \max\{2C+3, 2A\lambda^{(\log\lambda)^{2}}\}$, then the inequality $$\left|L\left(E_{1}\right)-L\left(E_{2}\right)\right| \leq C_\lambda |E_1 - E_2|^{\frac{c}{(\log\lambda)^{3}}}$$ holds for both cases, which completes the proof.
\hfill \qedbox
\medskip

From Theorem \ref{th01} and the Thouless formula:
\begin{equation}
L(z)=\int\log |E-z|\,dk(E),
\end{equation}
where $k(E)$ is the integrated density of states \cite{D}, we can directly draw the following corollary.

\begin{coro} For  sufficiently large $\lambda$,  the integrated density of states of $H(x)$ is H\"{o}lder continuous.
\end{coro}

\section{ Green's Function Estimate}

This section establishes upper bounds for the Green's function, which are crucial for analyzing eigenvalue problems and the decay properties of solutions.

For any interval $\Lambda  \subset \N$,  the restriction of $H$ to $\Lambda$ is given by
\begin{equation}
H_\Lambda = H_\Lambda(x) = R_\Lambda H(x) R_\Lambda,
\end{equation}
where $R_\Lambda$ is the restriction operator.

For any $E$ that is not an eigenvalue of $H_\Lambda(x)$, the Green's function is denoted by
\begin{equation}
G_\Lambda = G_\Lambda(x,E) = (H_\Lambda(x) - E )^{-1}.
\end{equation}

Let $\Lambda = [0,n]$ and define $H_n(x) = H_{[0,n]}(x)$. Then, for $0 \leq n_1 \leq n_2 \leq n$, it follows from Cramer's rule that
\begin{equation} \label{g}
G_n(x,E)(n_1,n_2) = \frac{ \det[H_{n_1-1}(x) - E] \det[H_{n-n_2-1}(2^{n_2+1}x) - E] }{ \det[H_n(x) - E] }.
\end{equation}

According to \cite{BG2}, the transfer matrices $A_n(x,E)$ defined in \eqref{a} can also be expressed as
\begin{align}
A_n(x,E) &= \begin{pmatrix}
 \det[H_n(x) - E] & -\det[H_{n-1}(2x) - E] \\
 \det[H_{n-1}(x) - E] & -\det[H_{n-2}(2x) - E]
\end{pmatrix}\label{an}.
\end{align}

Studies such as \cite{ADZ} and \cite{BoS} have established large deviation estimates for functions defined on the doubling map. We adapt the regularity conditions from Lemma 8.1 of \cite{BoS} to incorporate these estimates into our framework.

For any integer $n \geq 0$, we define the set
\begin{equation}
\Omega^{(n)}_1 := \{ x \in \mathbb{T} : 2^n x \equiv 0 \pmod{1} \}.
\end{equation}
Based on the assumption of $f$ in Section 2,  it follows that the discontinuity points of the functions $f$, $f\circ T$, $\cdots$, $f\circ T^n$ are all contained in the set $\Omega^{(n)}_1$. Since this set consists of a finite number of points on the torus, its Lebesgue measure satisfies $\operatorname{mes}(\Omega^{(n)}_1) = 0$ for all $n \in \mathbb{Z}_+$.

\begin{lemma}\label{le61}
Let $M > 1$, $n\in \N$, and let $F : \mathbb{T} \to \mathbb{R}$ satisfy:

\begin{enumerate}[(1)]
    \item $|F(x)| \leq 1$ for all $x \in \mathbb{T}$.

    \item For each  $j = 0, 1, \ldots, 2^{n} - 1$, the function $F$ is continuous on the interval
          \[
          I_{n,j} := \left[ \frac{j}{2^{n}}, \frac{j+1}{2^{n}} \right).
          \]

    \item $
          |F'(x)| \leq M \text{ for } x \in \mathbb{T} \setminus \Omega^{(n)}_1$, and
          $|F_{+}'(x)| \leq M \text{ for } x \in \Omega^{(n)}_1.
          $
\end{enumerate}
 Then we have
\begin{align}
{\rm mes} \Big[ x\in \T: \big| \int_{\mathbb{T}}Fdx-\frac{1}{r}\sum_{k=1}^rF(2^kx)\big|\,>\delta \Big]<e^{\frac{-c\delta ^2r}{J^2}}.
\end{align}
where $J >  \max\{\log (CM\delta^{-1}),n\}$.
\end{lemma}

\Proof We still follow Bourgain's method  for constructing martingale difference sequences.

Consider the family of conditional expectation operators $\{\mathbb{E}_{i}\}_{i\geq0}$, where $\mathbb{E}_{i}$ corresponds to the dyadic partition of  $\mathbb{T}$ into $2^{i}$ congruent intervals. That is,
\begin{equation}
\mathbb{E}_{i}[F] = \sum_{i} \left( \frac{1}{|I_i|} \int_{I_i} F dx \right) \chi_{I},
\end{equation}
where  $I_i:=\Big[\frac{j}{2^i},\,\frac{j+1}{2^i}\Big)$ and $j=0,1,\cdots,2^i-1$.

Express the function $F$ in the form of a martingale difference sequence:
\begin{equation}
F = \int_{\mathbb{T}} F \, dx + \sum_{j=1}^{\infty} \Delta_{j} F , \quad \text{where } \Delta_{j} F = \mathbb{E}_{j}[F ] - \mathbb{E}_{j-1}[F ].
\end{equation}
Noting that the original finite derivative assumption  in the Lemma 8.1 of \cite{BoS} was used to control $\| F - \mathbb{E}_{J} F \|_{\infty}$. Choose $J\geq n$, all points in set $\Omega^{(n)}_1$ are at the endpoints of interval $I_J$, and we have  that
\begin{align}
\| F - \mathbb{E}_{J} F \|_{\infty} &\leq \sup_{x\in\T}\left| \frac{1}{|I_J^{x}|}\int_{I_J^{x}}F(x)-F(y)\, dy\right|< M 2^{-J} < \frac{\delta}{10}.
\end{align}
where  $I_J^{x}$ denote the dyadic interval of length $2^{-J}$ that contains $x$.  Hence, we choose an integer $J>  \max\{\log (CM\delta^{-1}),n\}$, i.e. satisfies
\begin{equation}
 2^{J} > 10M\delta^{-1}\quad \text{and } \quad J\geq n.
\end{equation}
where $C$ is a suitable constant.

The remainder of the argument proceeds via Bourgain’s method, thus  establishing this lemma.
\hfill \qedbox
\medskip

\begin{prop}\label{le62}
For all $E\in [-2\lambda,2\lambda]$, there exists a large $n_0=n_0(x,\lambda)$ such that the set
\begin{align}\label{p11}
\mathcal{G} := \left\{x\in\mathbb{T}:\forall n\geq n_0,\sup_{\substack{|E|\leq2\lambda\\ 0\leq k_0\leq n^{8}}}\left|L^A_n(E)-\frac{1}{n^{4}}\sum_{k=1}^{n^{4}}\frac{1}{n}\log\left\|A_n\left(2^{k+k_0}x, E\right)\right\|\right|\leq1\right\}
\end{align}
satisfying $\mathcal{G} \subset \mathbb{T}$ with ${\rm mes}[\mathbb{T} \setminus  \mathcal{G}] = 0$.
\end{prop}

\Proof
 Fix any $E\in[-2\lambda,2\lambda]$ and let the function
\begin{equation}
F(x)=\frac{1}{  n\lambda}\log\left\|A_{n}(x, E)\right\|,
\end{equation}
where $x \in \mathbb{T}$.
 Corollary \ref{co01} (b)  implies that  $|F(x)| \leq 1$ for all $x \in \mathbb{T}$.
 It then follows from the assumption on $f$ (Section 2) and the transfer matrice \eqref{a} that $F$ is continuous on each interval $ I_{n,j} = \left[ \frac{j}{2^{n}}, \frac{j+1}{2^{n}} \right)$, $j = 0, 1, \ldots, 2^{n} - 1$.

For $x\in\T\setminus \Omega^{(n)}_1$, we can obtain that
\begin{align}\label{p14}
\left|F'(x) \right|&=\left|\frac{\frac{\partial\| A_{n}\|}{d x} (x, E)}{\lambda n\|A_{n}(x, E)\|}\right|\leq\frac{\|\frac{\partial A_{n}}{\partial x} (x, E)\|}{\lambda n\|A_{n}(x, E)\|}\notag\\
& =\Big\|\sum_{j=1}^n\Big(\prod_{k=j+1}^nA(T^{k}x, E)\Big(\begin{matrix}
-\lambda 2^{j}f'(T^{j}x) & 0\\
0 & 0\end{matrix}\Big)\prod_{k=1}^{j-1}A(T^{k}x, E)\Big)\Big\| \Big/ \Big(\lambda n\|A_{n}(x, E)\|\Big) \notag\\
& =\sum_{j=1}^n\Big(\prod_{k=j+1}^n\Big\|A(T^{k}x, E)\Big\|\Big\|\Big(\begin{matrix}
-\lambda2^{j} f'(T^{j}x) & 0\\
0 & 0\end{matrix}\Big)\Big\|\prod_{k=1}^{j-1}\Big\|A(T^{k}x, E)\Big)\Big\| \cdot \frac{1}{\lambda n}  \notag\\
&\leq \sum_{j=1}^n\Big[(|E|+\max_x|\lambda f(x)|+1)^{n-1} \cdot \lambda 2^{j}f'(T^{j}x) \Big]\cdot \frac{1}{\lambda n}\notag\\
&\leq (|E|+\max_x|\lambda f(x)|+1)^{n}\cdot\sum_{j=1}^n 2^{j} <(16\lambda)^{n}\doteq M.
\end{align}
The third step uses  $\|A_{n}(x, E)\|\geq 1$ to simplify the denominator,  the fourth step employs the estimate \eqref{amax}, and the fifth step relies on  the assumption of $f$ in Section 2.

Furthermore,  for the fixed  $E\in[-2\lambda,2\lambda]$,  since each $A (T^{j}\cdot ,E)$ ($j=1,2,\cdots,n$) is right-differentiable at the point $x\in \Omega^{(n)}_1$, the product rule implies that  $A_{n}(\cdot,E)$ is right-differentiable at $x$. Moreover, the invertibility of  $A_{n}(\cdot,E)$ ensures that its norm $\|A_{n}(\cdot,E)\|$ is also right-differentiable at  $x\in \Omega^{(n)}_1$, because the norm is differentiable at invertible operators. Consequently, by the chain rule, the composite function $F$ is  right-differentiable at $x\in \Omega^{(n)}_1$, and we have
\begin{align}\label{p141}
\left|F_+'(x)\right|&=\left|\frac{\frac{\partial_+ \|A_{n}\|}{\partial x} (x, E)}{\lambda n\|A_{n}(x, E)\|}\right|\leq\frac{\|\frac{\partial_+  A_{n}}{\partial x} (x, E)\|}{\lambda n\|A_{n}(x, E)\|}\notag\\
& =\Big\|\sum_{j=1}^n\Big(\prod_{k=j+1}^nA(T^{k}x, E)\Big(\begin{matrix}
-\lambda 2^{j}f_{+}'(T^{j}x) & 0\\
0 & 0\end{matrix}\Big)\prod_{k=1}^{j-1}A(T^{k}x, E)\Big)\Big\| \Big/ \Big(\lambda n\|A_{n}(x, E)\|\Big) \notag\\
&\leq (|E|+\max_x|\lambda f(x)|+1)^{n}\cdot\sum_{j=1}^n 2^{j} <(16\lambda)^{n}=M.
\end{align}

According to Lemma \ref{le61},  we choose an integer $J= n\log (C\lambda)$, and  the probability can be estimated as follows:
\begin{align}\label{p15}
&{\rm mes}\left[x\in\mathbb{T}:  \left|L^{A}_{n}(E)-\frac{1}{r}\sum_{k=1}^{r}\frac{1}{n}\log\left\|A_{n}\left(2^{k}x, E\right)\right\|\right|>\frac{1}{2}\right]\notag\\
=&{\rm mes}\left[x\in\mathbb{T}: \left|\frac{1}{ \lambda}L^{A}_{n}(E)-\frac{1}{r}\sum_{k=1}^{r}\frac{1}{n\lambda}\log\left\|A_{n}\left(2^{k}x, E\right)\right\|\right|>\frac{1}{2 \lambda}\right]<e^{-c\lambda^{-4}n^{-2}r}.
\end{align}

To consider the entire energy range $[-2\lambda\leq E \leq 2\lambda]$, the strategy of discretization approximation is employed.

Define the continuous bad set $B^{(n)}_{\text{cont}}=B^{(n)}_{\text{cont}}(x,\lambda) $:
\[
B^{(n)}_{\text{cont}}:=\left\{x\in T:\sup_{\substack{|E|\leq2\lambda\\ 0\leq k_0\leq r^{2}}}\left|L^A_{n}(E)-\frac{1}{r}\sum_{k=1}^{r}\frac{1}{n}\log\left\|A_{n}\left(2^{k+k_{0}}x, E\right)\right\|\right|>1\right\},
\]
and the discrete bad set $S_{E_{j},k_{0}}$:
\[
S^{(n)}_{E_{j},k_{0}}:=\left\{x\in T:\left|L^A_{n}(E_j)-\frac{1}{r}\sum_{k=1}^{r}\frac{1}{n}\log\left\|A_{n}\left(2^{k+k_{0}}x, E_j\right)\right\|\right|>\frac{1}{2}\right\},
\]
for fixed discrete energies $E_j$  and shift $k_0$.
According to \eqref{p15}, for any fixed $E_j$  and shift $k_0$, there exists a constant $c>0$ such that
\begin{equation}\label{eq:measure_estimate}
\operatorname{mes}\left(S^{(n)}_{E_{j},k_{0}}\right)<\exp(-c\lambda^{-4}n^{-2}r).
\end{equation}

We can obtain the that for all $x\in\T$ and $E\in[-2\lambda,2\lambda]$,
\begin{align}\label{p13}
\left|\frac{\partial\log\left\|A_{n}\right\|}{\partial E}(x, E) \right|&=\left|\frac{\frac{\partial \|A_{n}\|}{\partial E} (x, E)}{\|A_{n}(x, E)\|}\right|\leq\frac{\|\frac{\partial A_{n}}{\partial E} (x, E)\|}{\|A_{n}(x, E)\|}\notag\\
& =\Big\|\sum_{j=1}^n\Big(\prod_{k=j+1}^nA(T^{k}x, E)\Big(\begin{matrix}
1 & 0\\
0 & 0\end{matrix}\Big)\prod_{k=1}^{j-1}A(T^{k}x, E)\Big)\Big\| \Big/ \|A_{n}(x, E)\| \notag\\
&\leq (|E|+\max_x|\lambda f(x)|+1)^{n} <(4\lambda)^{n},
\end{align}
 and
\begin{align}\label{p131}
\left|\frac{d L^{A}_n}{d E} (E)\right|&\leq\int_{\mathbb{T}}\frac{1}{n}\left|\frac{\partial\log\left\|A_{n}\right\|}{\partial E}(x, E) \right|dx<(4\lambda)^{n}.
\end{align}

On the interval $[-2\lambda, 2\lambda]$ construct a finite point set $\{E_j\}_{j=1}^{J}$ that is $\frac{1}{4}\cdot (4\lambda)^{-n}$-dense. This means that for any $E\in I$, there exists some $E_j$ such that
\begin{equation}\label{ej}
|E-E_{j}|< \frac{1}{4}\cdot (4\lambda)^{-n}.
\end{equation}
The minimum number of grid points $J$ satisfies
\begin{equation}\label{j}
J\sim\frac{4\lambda}{\frac{1}{4}\cdot(4\lambda)^{-n}}=4\cdot(4\lambda)^{n+1}.
\end{equation}

 From the conditions \eqref{p13} and \eqref{p131}, it follows that for fixed $x$ and $k_{0}$,
\begin{equation}\label{eq5}
\left|\log\left\|A_{n}(2^{k+k_{0}}x,E)\right\|-\log\left\|A_{n}(2^{k+k_{0}}x,E')\right\|\right|\leq (4\lambda)^{n}|E-E'|,
\end{equation}
and
\begin{equation}\label{eq6}
|L^A_{n}(E)-L^A_{n}(E')|\leq (4\lambda)^{n}|E-E'|.
\end{equation}

We will prove that if $x$ belongs to the continuous bad set $B^{(n)}_{\text{cont}}$, then it must belong to some discrete bad set $S^{(n)}_{E_{j},k_{0}}$.

Take $x_0\in B^{(n)}_{\text{cont}}$. By the definition of $B^{(n)}_{\text{cont}}$, there exists some energy $E^{*}\in [-2\lambda, 2\lambda]$ and  $k_{0}^{*}\in\{0,1,\ldots,r^{2}\}$ such that
\begin{align}\label{623}
  \left|L^A_{n}(E^{*})-\frac{1}{r}\sum_{k=1}^{r}\frac{1}{n}\log\left\|A_{n}\left(2^{k+k_{0}^{*}}x_0, E^{*}\right)\right\|\right|>1.
\end{align}

By the denseness of the grid \eqref{ej}, there exists a grid point $E_{j^{*}}\in\{E_{j}\}$ satisfying
\begin{align}\label{eq8}
  |E^{*}-E_{j^{*}}|<\frac{1}{4}\cdot(4\lambda)^{-n} \quad\Rightarrow\quad (4\lambda)^{n}|E^{*}-E_{j^{*}}|<\frac{1}{4}.
\end{align}

 Consider the deviation at the discrete point $E_{j^{*}}$. We have
\begin{align*}
  &\left|L^A_{n}(E_{j^{*}})-\frac{1}{r}\sum_{k=1}^{r}\frac{1}{n}\log\left\|A_{n}\left(2^{k+k_{0}^{*}}x_0, E_{j^{*}}\right)\right\|\right| \\
  \geq &\left|L^A_{n}(E^{*})-\frac{1}{r}\sum_{k=1}^{r}\frac{1}{n}\log\left\|A_{n}\left(2^{k+k_{0}^{*}}x_0, E^{*}\right)\right\|\right|
   -\left|L^A_{n}(E^{*})-L^A_{n}(E_{j^{*}})\right| \\
  & -\left|\frac{1}{r}\sum_{k=1}^{r}\frac{1}{n}\log\left\|A_{n}\left(2^{k+k_{0}^{*}}x_0, E^{*}\right)\right\|-\frac{1}{r}\sum_{k=1}^{r}\frac{1}{n}\log\left\|A_{n}\left(2^{k+k_{0}^{*}}x_0, E_{j^{*}}\right)\right\|\right|.
\end{align*}

Substituting \eqref{623},  \eqref{eq5}, \eqref{eq6} and \eqref{eq8} into the above:
\begin{align*}
  \text{Left-hand side} &>1 - (4\lambda)^{n}|E^{*}-E_{j^{*}}| - (4\lambda)^{n}|E^{*}-E_{j^{*}}| \\
  & = 1 - 2 \cdot (4\lambda)^{n}|E^{*}-E_{j^{*}}|> \frac{1}{2}.
\end{align*}

Therefore,
\begin{align}
  \left|L^A_{n}(E_{j^{*}})-\frac{1}{r}\sum_{k=1}^{r}\frac{1}{n}\log\left\|A_{n}\left(2^{k+k_{0}^{*}}x_0, E_{j^{*}}\right)\right\|\right| > \frac{1}{2}.
\end{align}
This means $x_{0}\in S^{(n)}_{E_{j^{*}},k_{0}^{*}}$. Since $x_{0}$ is an arbitrary point in $B^{(n)}_{\text{cont}}$, we have  that
\begin{equation}\label{eq9}
B^{(n)}_{\text{cont}} \subseteq \bigcup_{j=1}^{J} \bigcup_{k_{0}=0}^{r^{2}} S^{(n)}_{E_{j}, k_{0}}.
\end{equation}

By the subadditivity of measure and the inclusion relation \eqref{eq9},
there exists a $\bar{n} _0=\bar{n}_0(\lambda)$ such that for
$r=n^{4}$ and $ n>\bar{n}_0$,
\begin{equation}\label{eq10}
\operatorname{mes}\left(B^{(n)}_{\text{cont}}\right) \leq \sum_{j=1}^{J} \sum_{k_{0}=0}^{r^{8}} \operatorname{mes}\left(S^{(n)}_{E_{j}, k_{0}}\right)
<Cr^{2}(4\lambda)^{n+1}e^{-c\lambda^{-4} n^{-2}r}<e^{-n}.
\end{equation}

Consider the sequence of bad sets $\{B^{(n)}_{\text{cont}}\}_{n > \bar{n}_0(\lambda)}$. From \eqref{eq10}, the sum of their measures converges:
\[
\sum_{n > \bar{n}_0(\lambda)}^{\infty} \operatorname{mes}(B^{(n)}_{\text{cont}}) < \sum_{n > \bar{n}_0(\lambda)}^{\infty} e^{-n} < \infty.
\]

According to the Borel-Cantelli lemma,  almost every phase $x$ belongs to only finitely many such bad sets $B^{(n)}_{\text{cont}}$. We define
\begin{align}
\mathcal{G} :&= \T \setminus \limsup_{n \rightarrow \infty} B^{(n)}_{\text{cont}}=\T \setminus\bigcap_{n_0>\bar{n}_0}^{\infty}\bigcup_{n=n_0}^{\infty}B^{(n)}_{\text{cont}}\notag\\
&= \{x\in \T: \exists \, n_0=n_0(x,\lambda)\in \mathbb{N} \text{ such that }   n_0>\bar{n}_0\text{ and } x\notin B^{(n)}_{\text{cont}} \text{ for all } n > n_0\}.\notag
\end{align}
By the Borel-Cantelli lemma, we have $\operatorname{mes}(\mathcal{G} ) = 1$. The set $\mathcal{G} $ is the full measure good set,
which is equivalent to the description in \eqref{p11}.
\hfill \qedbox
\medskip

Building on the settings in Proposition \ref{le62}, we now establish key estimates for the Green's function.

\begin{prop}\label{p2}
The Green's function  has the following properties:
\medskip

(a) For $x \in \mathcal{G}$ and $E\in[-2\lambda,2\lambda]$, we can find an integer $S\sim N^4$  with  $ N > 2n_{0}^{4}(x, \lambda)$ such that
\begin{equation}\label{g2}
|G_{\Lambda}(x,E)(n_{1},n_{2})|<e^{-|n_{1}-n_{2}|L(E)+C_aN}
\end{equation}
for all $n_{1}, n_{2} \in [S-\frac{N}{2}  , S+\frac{N}{2}]\doteq\Lambda$ and positive constant $C_a$.
\medskip

(b) Let $S$ and $N$ be integers satisfying  $S> N > 2n_{0}^{4}(x, \lambda)$. If for $x \in \mathcal{G}$ and $E\in[-2\lambda,2\lambda]$, there exists a  $C>0$ such that
\begin{equation}\label{p00}
\frac{1}{N} \log \|A_{N}(2^{k}x, E)\| \geq L^{A}_{N}(E) -CA
\end{equation}
holds for all $k\in [S-{N}  , S+{N}]$, then the following  holds:
\begin{align}
|G_{\Lambda}(x,E)(n_{1},n_{2})| <e^{-|n_{1}-n_{2}|L(E)+C_bN}
\end{align}
for all $n_{1}, n_{2} \in[S-\frac{N}{2}  , S+\frac{N}{2}]$ and positive constant $C_b$.
\end{prop}
\Proof
(a) According to Proposition \ref{le62}, for $x \in \mathcal{G}$ and $n>n_{0}^{4}(x,\lambda)$, let $n'=[  n^{\frac{1}{4}}] $ and $r'= [n/n' ] +1$. Then we have
\begin{equation}
\left\|A_{n}(x, E)\right\|\leq\prod_{k'=0}^{r'}\left\|A_{n'}\left(2^{k' n'}x, E\right)\right\|,
\end{equation}
By the  definition of $\mathcal{G}$, it follows that
\begin{align}
\frac{1}{n}\log\left\|A_{n}\left(2^{k_{0}}x, E\right)\right\|&\leq \frac{1}{n} \sum_{k'=0}^{r'} \log \|A_{n'}(2^{k' n'+ k_0}x, E)\|\notag\\&=\frac{k'r'}{n}\cdot\frac{1}{k'r'} \sum_{k'=0}^{r'} \log \|A_{n'}(2^{k' n'+ k_0}x, E)\|\notag\\
&\leq L^{A}_{n'}(E)+1+O(n^{-\frac{3}{4} } )
\end{align}
for $k_{0}\leq n^{2}$. For sufficiently large $n'$, Corollary \ref{co01} (b) implies that $L^{A}_{n'}(E) < L(E) + c$. Hence we have

\begin{equation}\label{ar}
\frac{1}{n} \log \| A_{n}(2^{k_0} x, E) \| < L(E) + C
\end{equation}
for a suitable positive constant $C$.
From \eqref{g} and \eqref{ar}, we can obtain that
\begin{align}\label{gr}
|G_{\Lambda}(x,E)(n_{1},n_{2})|
&=|G_{N}(2^{S-\frac{N}{2}}x,E)(n_{1},n_{2})| \notag\\
&\leq \frac{ \| A_{n_{1}-1}(2^{S-\frac{N}{2}}x) \| \, \| A_{N-n_{2}-1}(2^{S-\frac{N}{2}+n_{2}+1} x) \| }{ \left| \det[H_{N}(2^{S-\frac{N}{2}}x) - E] \right| } \notag\\
& < \frac{ e^{(N - |n_{1} - n_{2}|) L(E) + CN} }{ \left| \det[H_{N}(2^{S-\frac{N}{2}}x) - E] \right| }.
\end{align}
for all $n_{1}, n_{2} \in [S-\frac{N}{2}  , S+\frac{N}{2}]$.
To control the denominator, we relate it to the norm of the transfer matrices using  \eqref{an}:
\begin{align}
\|A_N(x,E) \|\leq \sqrt{2} \max_{(*_1,*_2)}\big\{
|\det[H_{N+*_1}(2^{*_2}x) - E] |\big\},
\end{align}
for ${(*_1,*_2)\in\{(0,0),(-1,0),(-1,1),(-2,1)\}}$.
Hence, we can further estimate \eqref{gr} for specified $(*_1,*_2)$ that
\begin{align}
|G_{N+*_1}(2^{S-\frac{N}{2}+*_2}x,E)(n_{1},n_{2})|&\leq
e^{(N+*_1-|n_{1}-n_{2}|)L(E)-\log\|A_{N}(2^{S-\frac{N}{2}+*_2}x,E)\|+\log\sqrt{2}+CN}\notag\\
 &<
e^{-|n_{1}-n_{2}|L(E)+[NL^{A}_{N}(E)-\log\|A_{N}(2^{S-\frac{N}{2}+*_2}x,E)\|]+2CN}
\end{align}
by using convergence of $L^{A}_N(E)$ for large $N$, and  the term $N|L^{A}_N(E)-L(E)|$, $*_1L(E)$ and $\log\sqrt{2}$ is absorbed into the $CN$-term.
Since the factor $*_1$  does not affect the estimation,  we can obtain that
\begin{align}\label{g3}
|G_{N}(2^{S-\frac{N}{2}+*_2}x,E)(n_{1},n_{2})|<
e^{-|n_{1}-n_{2}|L(E)+[NL^{A}_{N}(E)-\log\|A_{N}(2^{S-\frac{N}{2}+*_2}x,E)\|]+2CN}.
\end{align}

According to \eqref{p11}, it follows that
\begin{equation}
\left|L^{A}_{N}(E) - N^{-4}\sum_{k=N^{4}}^{2N^{4}}\frac{1}{N}\log\|A_{N}(2^{k}x, E)\|\right| \leq 1.
\end{equation}
Therefore, there exist  an integer $S_0\sim N^4$ such that
\begin{equation}\label{S0}
L^{A}_{N}(E) -\frac{1}{N}\log\|A_{N}(2^{S_0}x, E)\| \leq 1.
\end{equation}
Let $S=S_0-\frac{N}{2}+*_2$.  By selecting a siutable constant  $C_a>0$, it follows from \eqref{g3} and \eqref{S0} that
\begin{equation}
|G_{[\Lambda]}(x,E)(n_{1},n_{2})|<e^{-|n_{1}-n_{2}|L(E)+CN}
\end{equation}
for all $n_{1}, n_{2} \in [S-\frac{N}{2}  , S+\frac{N}{2}]$.

(b)  We have
$S-\frac{N}{2}+*_2\in [S-{N}, S+{N}]$ for $*_2\in\{0,1\}$ and large $N$. Building upon  \eqref{g3} and condition \eqref{p00}, it follows that
\begin{align}
|G_{N}(2^{S-\frac{N}{2}+*_2}x,E)(n_{1},n_{2})|&<
e^{-|n_{1}-n_{2}|L(E)+[NL^{A}_{N}(E)-\log\|A_{N}(2^{S-\frac{N}{2}+*_2}x,E)\|]+2CN}\notag\\
&\leq
e^{-|n_{1}-n_{2}|L(E)+CA+2CN}<e^{-|n_{1}-n_{2}|L(E)+3CN}.
\end{align}
Since the factor $*_2$  does not affect the estimation,  we can obtain that
\begin{align}
|G_{\Lambda}(x,E)(n_{1},n_{2})|
&=|G_{N}(2^{S-\frac{N}{2}}x,E)(n_{1},n_{2})| <e^{-|n_{1}-n_{2}|L(E)+3CN}.
\end{align}
By selecting an appropriate constant  $C_b>0$, we complete the proof.
\hfill \qedbox
\medskip

\section{ Double Resonance Set}

The double resonance set is the  critical ``bad set'' in localization proofs, representing parameter values that may disrupt  localization. This section demonstrates the exponential decay of the measure for double resonance sets.

First, leveraging the strong mixing property  of the underlying dynamics,  we present a useful lemma.

\begin{lemma}\label{le71}
Let $k=k(\lambda)$ be a large integer. For measurable sets $ I, S\subseteq  \T $,  we have
\begin{equation}
{\rm mes}\left[x\in I: 2^{k}x\in S \right]\leq\left(2^{k}\left|I\right|+1\right) 2^{-k}{\rm mes}\left[S\right].
\end{equation}
\end{lemma}

\Proof
 We cover  $I$ with subintervals of length $ 2^{-k} $. The minimum number of such subintervals needed to cover $I$ satisfies:

 \begin{equation}
l= \lceil 2^{k} |I| \rceil \leq 2^{k} |I| + 1.
\end{equation}

Let these  disjoint subintervals be $I^{(1)}, I^{(2)}, \dots, I^{(l)} $ with $I \subseteq \bigcup_{i=1}^{l} I^{(i)} $.
Since the strong mixing property of the dynamical system,   the independence of $x$ and $T^k x$ holds for large $k$.
Hence, we have
 \begin{equation}
{\rm mes}\left[x\in I^{(i)} : T^kx \in S\right] = {\rm mes}[I^{(i)}] \cdot {\rm mes}[S] = |I^{(i)}| \cdot {\rm mes}[S] \leq 2^{-k} {\rm mes}[S],
 \end{equation}
for $i=1,2,\cdots, l$.

Consider the event $P=\{x\in I : T^kx \in S\} $. By the additivity of probability, we obtain that
 \begin{equation}
{\rm mes}[P] \leq\sum_{i=1}^{l} {\rm mes}[P \cap I^{(i)}]  \leq \sum_{i=1}^{l} 2^{-k} {\rm mes}[S] \leq (2^{k} |I| + 1) 2^{-k} {\rm mes}[S].
\end{equation}
\hfill \qedbox
\medskip

Next, we estimate the measure of the double resonance set.

\begin{prop}\label{p3}
    For   sufficiently large $\lambda$,  let $N > N_0(\lambda)$ be a large integer. Define $\mathcal{D}^N $ as the set of  $x \in \mathbb{T}$ satisfying the following condition:  there exist  some choice of  $E$, $N_{1}$ and $k$, where  $E\in [-2\lambda,2\lambda]$,  $N_{1} \in [N^{12}, 2N^{12}] \cap \mathbb{Z}$  and $k \in [\bar{N}, 2\bar{N}] \cap \mathbb{Z}$   with $\bar{N} = [e^{(\log N)^{2}}]$ such that
\begin{equation}
\label{7.1}
\|G_{N_{1}}(x, E)\| > e^{N^{2}},
\end{equation}
\begin{equation}
\label{7.2}
\frac{1}{N} \log \|A_{N}(2^{k}x, E)\| < L^{A}_{N}(E) - CA,
\end{equation}
where $C>0$ is a suitable constant.
Then we have
\begin{equation}
\label{7.3}
{\rm mes}[\mathcal{D}^N] \leq e^{-\frac{cN}{\log\lambda}}
\end{equation}
for an appropriate constant $c>0$.
\end{prop}

\Proof To estimate the measure of set $\mathcal{D}^N $, we define

\begin{equation}\label{7.4}
\widetilde{\mathcal{D}}^N (x)= \bigcup_{N_{1} \sim N^{12}} \bigcup_{E \in \sigma(H_{N_{1}}(x)) } \left\{ y \in \mathbb{T} \mid \frac{1}{N} \log \|A_{N}(y, E)\| < L^{A}_{N}(E) -  CA\right\}.
\end{equation}
for a suitable constant $C$. Assume that $x \in {\mathcal{D}}^N$ satisfies \eqref{7.1} and \eqref{7.2}. From  $\|G_{N_{1}}(x, E)\|^{-1} = \operatorname{dist}(E, \sigma(H_{N_{1}}(x)))$, it follows that  the $e^{-N^{2}}$-errors of \eqref{7.1}  can be absorbed into the $CA$-term in \eqref {7.4} by the regularity \eqref{p13}, such that  $2^{k}x\in \widetilde{\mathcal{D}}^N(x)$ for  $k \sim \bar{N}$.

We denote the set on the right-hand side of \eqref{7.4} by $\mathcal{C}_{N}(E)$ and let
 \begin{align}\widetilde{\mathcal{C}}_{N}(E)=\left\{ x \in \mathbb{T} \mid \frac{1}{N} \log \|A_{N}(x, E)\| < L^{A}_{N}(E) -\frac{C}{2}A-\frac{\log\lambda}{N} \right\}.
  \end{align}
 It follows from  Corollary \ref{co01} (a) that
\begin{equation}
{\rm mes}\left[\mathcal{C}_{N}(E)\right]\leq {\rm mes}[\widetilde{\mathcal{C}}_{N}(E)]<e^{-\frac{cN}{\log\lambda}},
\end{equation}
for the suitable constant $C$ and large $N $.

Let $\mathbb{T}=\bigcup_{j=1}^{J} I_{j}$, with $|I_{j}|\sim e^{-N^{2}}$ and fix some $x_{j}\in I_{j}$.  The regularity  \eqref{p14}  ensures that $ \mathcal{C}_{N}(E) $ is a regular measurable set, whose boundary  has zero measure. Hence, according to Lemma \ref{le71}, we have
\begin{equation}
{\rm mes}\left[x\in I_{j}: 2^{k}x\in \mathcal{C}_{N}(E)\right]\leq\left(2^{k}\left|I_{j}\right|+1\right) 2^{-k}{\rm mes}\left[\mathcal{C}_{N}(E)\right].
\end{equation}
For large $N$, we can obtain that
\begin{align}\label{rd}
&{\rm mes}\left[x\in\mathbb{T}: 2^{k}x\in\mathcal{D}^N(x)\text{ for some }k\sim\bar{N}\right]\notag\\
&\leq\sum_{k\sim\bar{N}}\sum_{j=1}^{J}{\rm mes}\left[x\in I_{j}: 2^{k}x\in\mathcal{D}^N\left(x_{j}\right)\right]
\notag\\
&\leq\sum_{k\sim\bar{N}}\sum_{j=1}^{J}\sum_{N_{1}\sim N^{12}}\sum_{E\in\sigma(H_{N_{1}}(x_{j}))}{\rm mes}\left[x\in I_{j}: 2^{k}x\in \mathcal{C}_{N}(E)\right]
\notag\\
&\leq\sum_{k\sim\bar{N}}\sum_{j=1}^{J}\sum_{N_{1}\sim N^{12}}\sum_{E\in\sigma(H_{N_{1}}(x_{j}))}\left(2^{k}\left|I_{j}\right|+1\right) 2^{-k}{\rm mes}\left[\mathcal{C}_{N}(E)\right]
\notag\\
&\leq 2\sum_{k\sim\bar{N}}\sum_{j=1}^{J} N^{100}|I_{j}| e^{-\frac{cN}{\log\lambda}}\leq 2\bar{N} N^{100}e^{-\frac{cN}{\log\lambda}}<e^{-\frac{cN}{2\log\lambda}}.
\end{align}
In the first step,  the $e^{-N^{2}}$-errors arising by passing ${\mathcal{D}}^N$ to $\widetilde{\mathcal{D}}^N(x)$ could slightly alter the constant  $C$ of $CA$-term in \eqref {7.4}. In the fourth step, we utilize the fact that $H_{N_{1}}(x_{j})$ is an $(N_1+1)\times (N_1+1)$ matrix, which implies it has exactly $N_1+1$ eigenvalues (counting multiplicities).
By selecting appropriate positive constants $c$ and $C$, we complete the proof of the proposition.
\hfill \qedbox
\medskip

\section{Localization for the Doubling Map  }

In this section, we show that   there exists a set $\Omega=\Omega(\lambda)$ with ${\rm mes}[\T\backslash \Omega]=0$ such that for every $x\in\Omega$, the operator $H(x)$ given by \eqref{2.1} displays Anderson localization. Hence, Theorem \ref{th1} is a direct consequence of the following theorem.

\begin{theorem}\label{th3}
For sufficiently large $\lambda > 0$, let $\Omega = \Omega( \lambda) = \mathcal{G} \setminus \lim \sup  \mathcal{D}^N$ with ${\rm mes}(\mathbb{T} \setminus \Omega) = 0$, where $\mathcal{G}$ was defined in Proposition \ref{le62} and $\mathcal{D}^N$ in Proposition \ref{p3}.
For every $x \in \Omega$ the operator $H_\lambda(x)$ defined in \eqref{2.1} on $\ell^2(\N)$ with a Dirichlet boundary condition $u_{-1} = 0$ has pure point spectrum  and  all eigenfunctions decay exponentially at infinity.
\end{theorem}

\Proof
 Fix a sufficiently large $\lambda$ and an arbitrary $x \in \Omega$.
 By Theorem \ref{th0}, for any $E \in \mathcal{E}$, the corresponding generalized eigenfunction $u(x) = \{u_n(x)\}_{n\in\N}$ with $u_0=1$ is polynomially bounded:
 \begin{equation}
|u_n(x)| \leq C(1 + |n|)^{\delta}<2C|n|^{\delta}.
\end{equation}
Therefore, it suffices to prove that for any such $E$, the corresponding generalized eigenfunction $u(x)$ decays exponentially at infinity.

To this end, fix some such $E$ and $u(x)$, and choose $N=N(\lambda)$ so large  that $x\in\Omega^N\doteq \mathcal{G} \setminus \mathcal{D}^N$. Then, we prove that $u(x)$ decays exponentially.

\subsection*{Part 1: There exists an integer $N_1\sim N^{12}$ such that $\| G_{N_{1}} (x, E) \| > e^{N^2}$.}

On the one hand, since the generalized eigenfunction  $u(x)$ satisfies $H(x)u(x) = Eu(x)$, we have
\begin{align}
(H_{N_{1}} - E) R_{N_{1}} u(x)&= - R_{N_{1}} H R_{\mathbb{Z}_{+} \setminus [0,N_{1}]} u(x),\label{g1} \\
| (H_{N_{1}} - E) R_{N_{1}} u(x) | &= | u_{N_{1}+1}(x) |,
\end{align}
where $R_{N_{1}}=R_{[0,N_{1}]} $ is the restriction.  Since $u_0=1$, we obtain
\begin{equation}\label{gn}
\| G_{N_{1}} (x, E) \| > | u_{N_{1}+1}(x) |^{-1} .
\end{equation}

On the other hand,  define an interval
$\Lambda_1 = \left[n- \frac{N^{3}}{2}, n + \frac{N^{3}}{2}\right] = [a_1, b_1]$ for $n\sim N^{12}$.

Based on Proposition \ref{p2} (a) and Corollary \ref{co01} (b), we can find an integer $n\sim N^{12}$  such that
\begin{align}
 |G_{\Lambda_1}(x, E+\epsilon)(n,m)|<e^{-|n'-m'|L(E)+C_aN^3}<e^{C_aN^3}
 \end{align}
 for all $n,m \in \Lambda_1$ and all $\epsilon>0$ (to avoid the fixed $E$ as the  eigenvalue of $H_{\Lambda_1}(x)$).
 Therefore,  the Green's function is uniformly bounded:
  \begin{align}\label{gl1}
\| G_{\Lambda_1}(x, E+\epsilon)\|=&\sup_{\|\nu \|=1}\| G_{\Lambda_1}(x, E+\epsilon)\nu  \| \notag\\\leq& \sup_{m}\sum_{n}\left| G_{\Lambda_1}(x, E+\epsilon)(n, m)\right|\notag\\<&{N^3}e^{C_aN^3}< e^{2C_aN^3}.
\end{align}
Hence, the fixed $E$  is not an eigenvalue of $H_{\Lambda_1}(x)$. Since the Green’s function is continuous at $E$, taking the limit  $\epsilon\rightarrow 0$, the following estimate holds:
\begin{align}
 |G_{\Lambda_1}(x, E)(n,m)|<e^{-|n-m|L(E)+C_aN^3}
 \end{align}
 for all $n,m \in \Lambda_1$.

Combining with \eqref{g1} and the polynomial bound on $u(x)$, we can estimate the component of solution:
\begin{align}\label{uk}
\left|u_{n}(x) \right| &= \left|\left(H_{\Lambda_1 }(x)-E\right)^{-1}\left(u_{a_1-1}(x) \delta_{a_1} + u_{b_1+1}(x) \delta_{b_1}\right)(n)\right|\notag\\
&\leq \left|G_{\Lambda_1}(x, E)(n,a_1)\right|\left|u_{a_1-1}(x) \right|+ \left|G_{\Lambda_1 }(x, E)(n, b_1)\right|\left|u_{b_1+1}(x) \right|\notag\\
& \leq 4C N^{12{\delta}}e^{- \frac{1}{2} N^{3}\log\lambda}  < e^{-N^{2}}.
\end{align}

Taking $ N_1 =n -1\sim N^{12}$,  we can  obtain  $\| G_{N_{1}} (x, E) \| > e^{N^2}$ from \eqref{gn} and \eqref{uk}.

\subsection*{Part 2: The generalized eigenfunction $u(x)$  decays exponentially.}

In view of the definition of $\Omega^{N}$, we conclude that the converse of \eqref{7.2} holds for every $k \in \left[\frac{3\bar{N}}{2}- \frac{\bar{N}}{2}, \frac{3\bar{N}}{2}+ \frac{\bar{N}}{2}\right]=[\bar{N}, 2\bar{N}]$.

Define an interval
$\Lambda_2 = \left[\frac{3\bar{N}}{2}- \frac{\bar{N}}{4}, \frac{3\bar{N}}{2}+ \frac{\bar{N}}{4}\right]=\left[\frac{5\bar{N}}{4}, \frac{7\bar{N}}{4}\right]=[a_2, b_2]$. According to
Propsition \ref{p2} (b), we have
\begin{align}\label{gl}
 |G_{\Lambda_2}(x, E+\epsilon)(n,m)|<e^{-|n-m|L(E)+C_b\bar{N}}
 \end{align}
 for all $n,m \in \Lambda_2$ and all  $\epsilon>0$, and  we can obtain the following uniform bound:
 \begin{align}\label{9g1}
\| G_{\Lambda_2}(x, E+\epsilon)\|=&\sup_{\|\nu \|=1}\| G_{\Lambda_2}(x, E+\epsilon)\nu  \| \notag\\\leq& \sup_{m}\sum_{n}\left| G_{\Lambda_2}(x, E+\epsilon)(n, m)\right|\notag\\<&\bar{N}e^{C_b\bar{N}}< e^{2C_b\bar{N}}.
\end{align}
This indicates that  the fixed $E$ is not an eigenvalue of $H_{[\bar{N}, 2\bar{N}]}(x)$. Letting  $\epsilon\rightarrow 0$, it follows that
\begin{equation}\label{9.25}
\left|G_{\Lambda_2}(x, E)(n, m)\right|  <e^{-|n-m|L(E)+C_b\bar{N}}.
\end{equation}

Combining with \eqref{9.25} and the polynomial bound on the solution $u(x)$, we obtain
\begin{align}
\left|u_{n}(x) \right| &\leq \left|\left(H_{\Lambda_2 }(x)-E\right)^{-1}\left(u_{a_2-1}(x) \delta_{a_2} + u_{b_2+1}(x) \delta_{b_2}\right)(n)\right|\notag\\
&\leq \left|G_{\Lambda_2}(x, E)(n,a_2)\right|\left|u_{a_2-1}(x) \right|+ \left|G_{\Lambda_2 }(x, E)(n, b_2)\right|\left|u_{b_2+1}(x) \right|\notag\\
& \leq 4C \bar{N}^{{\delta}}e^{- \frac{1}{16} \bar{N}\log\lambda}  < e^{-cn\log\lambda}
\end{align}
for any $n \in [\frac{11\bar{N}}{8}, \frac{13\bar{N}}{8}]$ and a suitable constant $c>0$.

Having shown that the key condition is satisfied for all sufficiently large $N$ in Proposition \ref{p2} and Proposition \ref{p3}, we obtain that the generalized eigenfunction $u(x)$  decays exponentially at
infinity.

Since the above results hold for arbitrarily large $\lambda$ and for arbitrary $E\in\mathcal{E}$ and $x\in\Omega $,  the proof of Theorem \ref{th3} is complete.
 \hfill \qedbox

\medskip
\vskip1cm
\noindent{$\mathbf{Acknowledgments}$}

This work was supported by the NSFC (grant no. 11571327 and 11971059).


\vskip1cm
\begin{thebibliography}{[aa]}

\bibitem{A} Avila, A.: \textit{Global theory of one-frequency Schr\"{o}dinger operators.} Acta Math. 21(1), 1-54 (2015)

\bibitem{AJ} Avila, A., Jitomirskaya, S.: \textit{The Ten Martini problem.} Ann. of Math. 170, 303-342 (2009)

\bibitem{ADZ} Avila, A., Damanik, D., Zhang, Z.: \textit{Schr\"{o}dinger operators with potentials generated by hyperbolic transformations: I-positivity of the Lyapunov exponent.} Invent. Math. 231, 851-927 (2023)


\bibitem{Ba}  Baladi, V.: \textit{Positive Transfer Operators and Decay of Correlations.} World Scientific (2000)

\bibitem{B}  Bjerkl\"{o}v, K.: \textit{Positive Lyapunov exponent for some Schr\"{o}dinger cocycles over strongly expanding  circle endomorphisms.} Comm. Math. Phys. 379, 353-360 (2020)

\bibitem{Bou}  Bourgain,  J.: \textit{Green's function estimates for lattice Schr\"{o}dinger operators and applications.} Princeton, NJ: Princeton University Press, (2005)

\bibitem{BB} Bourgain, J., Bourgain-Chang, E.: \textit{A note on Lyapunov exponents of deterministic strongly mixing  potentials.} J. Spectr. Theory 5, 1-15 (2015)

\bibitem{BG} Bourgain, J., Goldstein, M.: \textit{On nonperturbative localization with quasi-periodic potential.} Ann. of Math. 152, 835-879 (2000)

\bibitem{BG2}   Bourgain, J.,  Goldstein, M.,  Schlag, W.: \textit{Anderson localization for Schr\"{o}dinger operators on $\Z$ with potentials given by the skew-shift.} Comm. Math. Phys. 220.3 583621, (2001).

\bibitem{BJ}  Bourgain, J., Jitomirskaya S.: \textit{Continuity of the Lyapunov exponent for quasiperiodic operators with analytic potential.} J. Stat. Phys. 108 (2002), 1203-1218.

\bibitem{BoS} Bourgain, J., Schlag, W.: \textit{Anderson localization for Schr\"{o}dinger operators on $\mathbb{Z}$ with strongly mixing  potentials.} Commun. Math. Phys. 215(1), 143-175 (2000)

\bibitem{BrS} Brin, M., Stuck, G.: \textit{Introduction to dynamical systems.} Cambridge University Press, Cambridge (2015)

\bibitem{CL} Carmona, R., Lacroix, J.: \textit{Spectral theory of random Schr\"{o}dinger operators.} Boston: Birkh\"{a}user (1990)

\bibitem{CS} Chulaevsky, V., Spencer, T.: \textit{Positive Lyapunov exponents for a class of deterministic potentials.} Commun. Math. Phys. 168, 455-466 (1995)

\bibitem{D} Damanik, D., Fillman, J.: \textit{One-dimensional ergodic Schr\"{o}dinger operators I. General theory}.  Graduate Studies in Mathematics Vol. 221, American Mathematical Society, Providence (2022)

\bibitem{D05}  Damanik, D., Killip, R.: \textit{Almost everywhere positivity of the Lyapunov exponent for the doubling  map}. Commun. Math. Phys. 257, 287-290 (2005)

\bibitem{D07}  Damanik, D.: \textit{Lyapunov exponents and spectral analysis of ergodic Schr\"{o}dinger operators: a survey  of Kotani theory and its applications. In: Spectral Theory and Mathematical Physics: a festschrift  in honor of Barry Simon's 60th birthday}, Proc. Sympos. Pure Math. 76, Part 2, Amer.  Math. Soc., Providence, RI, 539-563, (2007)

\bibitem{D16} Damanik, D., Fillman, J., Lukic, M., Yessen, W.: \textit{Characterizations of uniform hyperbol
icity and spectra of CMV matrices.} Discrete Contin. Dyn. Syst. Ser. S., 9, 1009-1023,  (2016)

\bibitem{D23}Damanik, D., Fillman, J.: \textit{The almost sure essential spectrum of the doubling map model is connected.} Comm. Math. Phys. 400, 793-804 (2023)

\bibitem{FP} Figotin, A., Pastur, L.: \textit{Spectra of random and almost-periodic operators.} Grundlehren der Math ematischen Wissenschaften 297, Berlin: Springer-Verlag (1992)

\bibitem{F}  Forman, Y. M.: \textit{Localization and cantor spectrum for quasiperiodic discrete Schr\"{o}dinger operators with asymmetric, smooth, cosine-like sampling functions.} Dissertation,  Yale University (2022).

\bibitem{GS}  Goldstein M.,  Schlag, W.: \textit{H\"{o}lder continuity of the integrated density of states for quasi-periodic Schr\"{o}dinger equations and averages of shifts of subharmonic functions.} Ann. of Math. 154, 155-203 (2001).


\bibitem{H} Herman, M.: \textit{Une m\'{e}thode pour minorer les exposants de Lyapounov et quelques exemples montrant  le caract\`{e}re local d'un th\'{e}or\`{e}me d'Arnold et de Moser sur le tore de dimension 2.} Comment. Math.  Helv. 58, 453-502 (1983)

\bibitem{JK} Jitomirskaya, S., Kachkovskiy, I.: \textit{All couplings localization for  quasiperiodic operators with monotone potentials.} J. Eur. Math. Soc. 21, 777-795 (2018)

\bibitem{LT} Lagendijk, A., Tiggelen, B., Wiersma, D.: \textit{Fifty years of Anderson localization.} Phys. Today 62(8), 24-29 (2009)

\bibitem{LP} Lin, Y., Piao, D., Guo, S.: \textit{ Anderson localization for the quasi-periodic CMV matrices with Verblunsky coefficients defined by the skew-shift}. J. Funct. Anal. 284(1), 1-25 (2023)

\bibitem{Sh}Shnol, I.E.: \textit{On the behavior of eigenfunctions.} (Russian), Doklady Akad. Nauk SSSR (N.S.) 94, 389-392 (1954)

\bibitem{Si}Simon, B.: \textit{Spectrum and continuum eigenfunctions of Schr\"{o}dinger Operators.} J. Funct. Anal. 42, 66-83 (1981)

\bibitem{SS} Sorets, E., Spencer, T.:  \textit{Positive Lyapunov exponents for Schr\"{o}dinger operators with quasi-periodic potentials.} Comm. Math. Phys. 142, 543-566 (1991)

\bibitem{WZ}  Wang, Y., Zhang, Z.: \textit{Uniform positivity and continuity of Lyapunov exponents for a class of $C^2$ quasiperiodic Schr\"{o}dinger cocycles.} J. Funct. Anal. 268, 2525-2585 (2015)

\bibitem{Y}  Young, L.: \textit{Some open sets of nonuniformly hyperbolic cocycls.} Ergodic Theory Dynam. Sys. 13, 409-415 (1993)

\bibitem{ZL} Zhang, G., Li, X.: \textit{Positive Lyapunov exponent for some Schr\"{o}dinger cocycles over multidimensional strongly expanding torus endomorphisms.} Nonlinearity 36, 401-425 (2023)

\bibitem{Z12} Zhang, Z.: \textit{Positive Lyapunov exponents for quasiperiodic Szeg\H{o} cocycles.} Nonlinearity 25, 1771-1797 (2012)

\bibitem{Z24} Zhang, Z.: \textit{Uniform Positivity of the Lyapunov Exponent for  Monotone Potentials Generated by the Doubling Map.} Comm. Math. Phys. 405: 231 (2024)



\end{thebibliography}
\end{document}